\documentclass[12pt]{article}
\usepackage{newinutile-aac,siunitx}
\usepackage{amsmath,amsfonts,amsthm,amssymb}
\usepackage{epsfig,graphics,color,calc,graphicx}
\usepackage{cite}
\usepackage{mathrsfs}

\newtheorem{theorem}{Theorem}
\newtheorem{corollary}[theorem]{Corollary}

\newtheorem{lemma}[theorem]{Lemma}
\newtheorem{remark}[theorem]{Remark}
\newtheorem{definition}[theorem]{Definition}
\newtheorem{proposition}[theorem]{Proposition}

\makeatletter
\newcommand{\boldcal}[1]{\boldsymbol{\mathcal{#1}}}
\newcommand{\textupmd}[1]{\textup{\textmd{#1}}}
\newcommand{\davector}[1]{#1}
\newcommand{\matrixfont}{\boldcal}
\newcommand{\damatrix}[1]{\matrixfont{#1}}

\newcommand{\datimecyl}[2]{\newcommand{#1}{{#2}_{\maxT}}}
\newcommand{\newvector}[2]{%
\expandafter\expandafter\expandafter\def\expandafter\csname%
genericvectbasename#1\endcsname{#2}
\expandafter\expandafter\expandafter\def\expandafter\csname%
gvectsc#1\endcsname{#2}
\expandafter\expandafter\expandafter\def\expandafter\csname%
gvect#1\endcsname{\davector{#2}}
}
\newvector{}{f}
\newvector{b}{g}
\newvector{c}{h}
\newcommand{\newmatrix}[2]{%
\expandafter\expandafter\expandafter\def\expandafter\csname%
genericmatrbasename#1\endcsname{#2}
\expandafter\expandafter\expandafter\edef\expandafter\csname%
genericmatrbasenameup#1\endcsname{\uppercase{#2}}
\expandafter\expandafter\expandafter\def\expandafter\csname%
gmatr#1\endcsname{\damatrix{\uppercase{#2}}}
\expandafter\expandafter\expandafter\edef\expandafter\csname%
gmatrel#1\endcsname{#2}}
\newmatrix{}{a}
\newmatrix{b}{b}
\newmatrix{c}{c}

\newcommand{\SpDim}{N}

\newcommand{\Oset}{\varOmega}

\newvector{normal}{\nu}

\newcommand{\Ballbasename}{B}
\newcommand{\Ball}[2]{ \ifthenelse{ \equal{#1}{*} \or \equal{#1}{\empty} }{%
            \Ballbasename(#2) }{ \ifthenelse{ \equal{#2}{*} \or \equal{#2}{\empty}}{%
            \Ballbasename_{#1} }{ \Ballbasename_{#1}(#2)}
            }
}

\newcommand{\const}{\gamma}

\newcommand{\sobx}[1][]{\ifthenelse{\equal{#1}{}}%
{\def\@sobx{1}}{\def\@sobx{#1}}%
\frac{\SpDim #1}{\SpDim - \@sobx}}
\newcommand{\isobx}[1][]{\ifthenelse{\equal{#1}{}}%
{\def\@sobx{1}}{\def\@sobx{#1}}%
\frac{\SpDim - \@sobx}{\SpDim #1}}

\newcommand{\maxT}{T}
\datimecyl{\Ocyl}{\Oset}

\newcommand{\@di}{\textupmd{d}}
\newcommand{\di}{\,\@di}
\newcommand{\grad}{\operatorname{\nabla}}

\newcounter{da@aux}
\newcounter{da@auxb}
\newcommand{\onlyifnotone}[1]{\ifnum\arabic{#1}=1\else\arabic{#1}\fi}
\newcommand{\der}[3][1]{
\ifthenelse{ \equal{#1}{1} }{ \def\@deraux{} }{ \def\@deraux{#1} }
\frac{\@di^{\@deraux}#2}
            {\@di #3^{\@deraux}}}
\newcommand{\pder}[2]{\setcounter{da@aux}{-1}\setcounter{da@auxb}{0}%
\def\Last@Tok{\relax}%
\def\List@Tok{\relax}%
\Check@Tok#2\relax%
\frac{\partial^{\onlyifnotone{da@auxb}} #1}{\List@Tok}}
\newcommand{\Check@Tok}[1]{\xdef\@Argom{#1}
\ifnum\arabic{da@aux}=-1 \setcounter{da@aux}{0}
\xdef\Last@Tok{\@Argom}\fi%
\ifx#1\relax \let\next=\relax%
\xdef\ProvList@Tok{\List@Tok}%
\xdef\List@Tok{\ProvList@Tok \partial\Last@Tok^{\onlyifnotone{da@aux}}}%
\addtocounter{da@auxb}{\value{da@aux}}%
\else%
\ifx\Last@Tok\@Argom\stepcounter{da@aux}%
\else
\xdef\ProvList@Tok{\List@Tok}%
\xdef\List@Tok{\ProvList@Tok \partial\Last@Tok^{\onlyifnotone{da@aux}}}%
\xdef\Last@Tok{\@Argom}%
\addtocounter{da@auxb}{\value{da@aux}}
\setcounter{da@aux}{1}%
\fi
\let\next=\Check@Tok
\fi
\next
}

\newcommand{\avint}[2][-]{\mathchoice
   {\@avint[d]{\displaystyle}{\textstyle}{#1}{#2}}%
   {\@avint{\textstyle}{\scriptstyle}{#1}{#2}}%
   {\@avint{\scriptstyle}{\scriptscriptstyle}{#1}{#2}}%
   {\@avint{\scriptscriptstyle}{\scriptscriptstyle}{#1}{#2}}%
}
\newcommand{\@avint}[5][*]{
    \ifx\ilimits@\displaylimits\def\@@avintchk{Y}\else
    \def\@@avintchk{N}\fi
     \ifthenelse{ \equal{#1}{d} \and
        \equal{\@@avintchk}{Y} }{%
     {\setbox0=\hbox{$#2{#3#4}{\int_{#5}}$}\kern .5\wd0
     \vcenter{\hbox{$#3#4$}}\kern-.5\wd0}\!\int_{#5}}{
     {\setbox0=\hbox{$#2{#3#4}{\int}$}\kern .5\wd0
     \vcenter{\hbox{$#3#4$}}\kern-.5\wd0}\!\int_{#5}}}

\newcommand{\abs}[1]{\lvert#1\rvert}
\newcommand{\Abs}[1]{\left|#1\right|}

\newcommand{\Pposbase}[3][*]{\ifthenelse{\equal{#1}{*}}%
{(#2)_{#3}}{\left(#2\right)_{#3}}}

\newcommand{\@pint}[2]{(#1,#2)}
\newcommand{\pint}[3][*]{\ifthenelse{\equal{#1}{*}}{\def\@daHil{}}%
{\def\@daHil{#1}}{\@pint{#2}{#3}}_{\@daHil}}

\newcommand{\Lspbasename}{L}
\newcommand{\Lsp}[3][\relax]{\ifthenelse{\equal{#3}{\empty}}%
{{{\Lspbasename}^{#2}}}{{{\Lspbasename}^{#2}}#1(#3#1)}}

\newcommand{\Cspbasename}{C}
\newcommand{\Csp}[3][\relax]{\ifthenelse{\equal{#3}{\empty}}%
{{{\Cspbasename}^{#2}}}{{{\Cspbasename}^{#2}}#1(#3#1)}}

\DeclareMathOperator{\Div}{div}

\newcommand{\norm}[1]{\lVert#1\rVert}
\newcommand{\Norm}[1]{\left\lVert#1\right\rVert}
\newcommand{\norma}[2]{\norm{#1}_{#2}}
\newcommand{\Norma}[2]{\Norm{#1}_{#2}}
\newcommand{\Measbase}[2][*]{\ifthenelse{\equal{#1}{*}}%
{\lvert#2\rvert}{\left|#2\right|}}

\newcommand{\Hareabase}[2][*]{\ifthenelse{\equal{#1}{*}}%
{\lvert#2\rvert_{\SpDim-1}}{\left|#2\right|_{\SpDim-1}}}

\newcommand{\scpr}{\cdot}

\newcommand{\oo}{\infty}

\newcommand{\eps}{\varepsilon}
\newcommand{\phj}{\varphi}

\newcommand{\Om}{\varOmega}
\newcommand{\diracm}{\@ifnextchar*{\diracm@}{\diracm@@}}
\def\diracm@*#1{\diracm@@{\{#1\}}}
\newcommand{\diracm@@}[1]{\delta_{#1}}

\newlength{\Indent}
\newlength{\Parskip}
\ifx\normalparindent\UnDeFiNeD\else
\fi

\newcommand{\ignore}[1]{\relax}

\newcommand{\daUndef}[1]{\let#1\UnDeFiNeD}

\newcommand{\tshift}[1]{\widetilde{#1}}

\newcommand{\ve}{v_{\eps}}
\newcommand{\aun}{a_1}
\newcommand{\aune}{\aun^\eps}
\newcommand{\adu}{a_2}
\newcommand{\adue}{\adu^\eps}
\newcommand{\bun}{b_1}
\newcommand{\bune}{\bun^\eps}
\newcommand{\bdu}{b_2}
\newcommand{\bdue}{\bdu^\eps}
\newcommand{\Bm}{B}
\newcommand{\Bme}{\Bm^\eps}
\newcommand{\bp}{b}
\newcommand{\bpe}{\bp^\eps}
\newcommand{\fpco}{\Big(\bdue+\varepsilon\frac{\bpe}{\bune}\Big)}
\newcommand{\tsfpco}{\Big(\tshift{\bdue}
+\varepsilon\frac{\tshift{\bpe}}{\tshift{\bune}}\Big)}
\newcommand{\fpcov}{\Big[\fpco \ve\Big]}
\newcommand{\tsfpcov}{\Big[\tsfpco \tshift{\ve}\Big]}

\newcommand{\Hper}{H^1_{\#}}

\makeatother

\newlength{\pecettawidth}
\begin{document}
\title{Diffusion in inhomogeneous media with periodic microstructures}

\authorc{Micol Amar\emailf{micol.amar@sbai.uniroma1.it}, 
Daniele Andreucci\emailf{daniele.andreucci@sbai.uniroma1.it},
Emilio N.M.\ Cirillo\emailf{emilio.cirillo@uniroma1.it}}
\affiliation{Dipartimento di Scienze di Base e Applicate per l'Ingegneria,
             Sapienza Universit\`a di Roma,
             via A.\ Scarpa 16, I--00161, Roma, Italy.}


\begin{abstract}
Diffusion in inhomogeneous materials can be described by both
the Fick and Fokker--Planck diffusion equations.
Here, we study a mixed Fick and Fokker--Planck diffusion
problem with coefficients rapidly oscillating both in
space and time.
We obtain macroscopic models performing the homogenization
limit by means of the unfolding technique.
\end{abstract}


\keywords{Diffusion; Fick's law; Fokker--Planck diffusion law;
homogenization.}



\maketitle

\renewcommand{\contentsname}{\vskip -1.2 cm $\phantom.$\par}
\addtocontents{toc}{\protect\setcounter{tocdepth}{-1}}
\tableofcontents
\addtocontents{toc}{\protect\setcounter{tocdepth}{3}}

\section{Introduction}
\label{s:introduzione}
\par\noindent
The study of the motion of particles diffusing in a confined
region is relevant in many different fields (see, for instance,
the recent papers
\cite{CCS2018,CKMS2016,CKMSS2016,DmiKL2019,dCGdAM2015,LCW2016}
and the references therein).
In several studies, it has been shown that the interaction of particles
with the
walls results into a diffusive coefficient depending on the
space coordinates \cite{HDL2006,LBLO2001}.
A rather natural microscopic counterpart is represented by the
random walk models, with hopping probabilities depending on
the site coordinates. Such kind of models have been, for instance,
introduced in the study of wetting phenomena, in which the effect of
competition between long range attraction and reflection
at the wall is modeled \cite{DcDH2008}.
We also mention that space dependent diffusion is also
considered in some biological ionic channel models, to justify
the selection of ionic species \cite{AAB2017bis,ABC2014}.

In the context of diffusion motion in inhomogeneous materials,
due to the space dependence of the diffusion coefficient,
the derivation of the macroscopic equation is not straightforward.
Indeed,
assuming that the flux is given either by $-B\nabla u$ or
$-\nabla (cu)$, where $B$ and $c$
represent the diffusion coefficient and $u$
the density field, gives rise to two different diffusion equations,
known in the literature as the Fick and the Fokker--Planck
diffusion laws \cite{L1984,MBCS2005,S2008,S1993,SBBP1990}, respectively.
In the recent paper \cite{ACCG2019}, relying on a
hydrodynamic limit computation, it has been proved that the two different choices
mentioned above for the flux are connected to the microscopic
structure of the inhomogeneity. Indeed, for local isotropic space
inhomogeneities, the Fokker--Planck version of the flux is found, whereas
when the space inhomogeneity is exclusively due to local anisotropy, the
Fick expression is recovered. In mixed situations, the general
flux structure $-B\nabla (cu)$ is found and the corresponding general diffusion
law $u_t=\nabla\cdot[B\nabla(cu)]$ is obtained.

Here,
we study such a mixed Fick and Fokker--Planck
diffusion problem for inhomogeneous materials, whose diffusion
properties are described by means of rapidly oscillating coefficients
with respect to both space and time (see
the initial--boundary value
problem \eqref{gpro010}--\eqref{gpro050} below).
We assume that such a material has an underlying periodic microstructure, whose characteristic length
is of order $\varepsilon^\alpha$ ($\varepsilon$ and $\alpha$ being strictly positive real
parameters), while its time oscillation has a period of order
$\varepsilon^\beta$, $\beta$ being another strictly positive parameter.

As usual in this kind of very fast oscillating problems, the main purpose is to
obtain a macroscopic model, overcoming the
difficulties due to the intricate original geometry and appearing, for instance, in the numerical approach.
To this purpose, we are led to let $\varepsilon \to 0$, thus performing a homogenization limit. The resulting
equation models the effective behavior of the medium in the macroscopic setting,
keeping memory, in general, of the underlying periodic structure.
However, the homogenization of the problem \eqref{gpro010}--\eqref{gpro050}
seems to be a too ambitious goal, without some further structural assumptions on the coefficients.
For this reason, we shall
confine our investigation to a particular case introduced
in Section~\ref{s:p-mod}, where the capacitive coefficient in front of the time-derivative
and the Fokker coefficient inside the spatial gradient are assumed to
have a separate dependence on the time and space oscillating variables
(similarly to the classical Fick case, treated in \cite{AAB2017}),
but admitting that the Fokker coefficient
can be perturbed by a non--product additional coefficient
of amplitude $\varepsilon$. We refer to this case as the \emph{weakly non--product case}.
However, as we will see in the sequel, when we consider the
pure Fokker--Planck model (i.e., the diffusion matrix in front of the gradient term is the identity), with unit
capacity, the perturbation does not play any role in the limit equation and disappears from the expression of
the effective coefficients (see Subsections \ref{s:np-mod01} and \ref{s:np-mod02}).

More precisely, in Subsection~\ref{s:np-mod01}, by using the well--known
two--scale expansion technique,
introduced in \cite{lions}, we formally show that the
non--product perturbation does not affect the upscaled equation, when its amplitude $\varepsilon$ is
of the same order of the spatial oscillation period $\varepsilon^\alpha$ (i.e. $\alpha= 1$),
as long as we assume the diffusion matrix $B=I$. This result is also rigorously proven in Section \ref{s:hom}.
It is rather natural to ask what
would happen if such a perturbation were more intense with
respect to the microscopic oscillation scale. This case will be considered in
Subsection \ref{s:np-mod02}.
However, since in the formal expansions we are obliged to deal only with integer powers,
we cannot consider an exponent smaller than one for the
$\varepsilon$ amplitude of the non--product perturbation.
Hence, we accelerate the microscopic spatial oscillations choosing
the smaller oscillation period $\varepsilon^2$.
We will show that also in this case the small non--product perturbation
does not affect the upscaled equation.
However, we do not propose this as a general conclusion,
since it could depend on the special choice of the diffusion matrix and the capacity coefficients.

In Section \ref{s:hom}, we will rigorously prove that
the same property holds also in the general mixed Fick and Fokker-Planck case, if the amplitude of the
non--product perturbation in the Fokker coefficient
is strictly smaller than the spatial oscillation period, i.e. $\alpha<1$.

At our knowledge, diffusion problems governed by Fick and/or Fokker--Planck
laws depending on capacitive, diffusive and Fokker coefficients,
highly oscillating with respect to
time and space simultaneously are
not considered in an extensive body of mathematical literature.
Among the few results, we recall \cite{AAB2017,AAB2017bis,AAGT2019,Floden:Holmbom:Lindberg:2012,Floden:Holmbom:Lindberg:Persson:2013,
Holmbom:1997,Holmbom:Svanstedt:Wellander:2005,Persson:2012}.

In particular, in \cite{AAB2017} the authors have considered a homogenization problem in
the framework of the standard heat equation, which is very close to the case analyzed in the present research.
The main novelty of that paper, with respect to former literature,
is not only the presence of a capacitive term, oscillating both in space and time, but also
the fact that the homogenization problem has been solved
under completely general assumptions on the space and
time microscopic oscillation periods, i.e.
the oscillation periods $\tau$ and $\varepsilon$, respectively,
for time and space variables, are completely independent.

In this respect, those authors had to distinguish between two
cases, namely, when the space period is smaller than or larger
than the square of the time period. These cases were called
\textit{fast} and \textit{slow} oscillations, respectively.

In the present paper, we shall have to distinguish between
these two situations, as well; however, we will confine our investigation
only to the case where time and space oscillation periods are powers
of the common small parameter $\varepsilon$ which, as recalled above, represents the perturbation
size.

The approach we follow here is essentially the same as the one adopted in
\cite{AAB2017} and it is based on the periodic unfolding homogenization technique,
first introduced in \cite{Cioranescu:Damlamian:Griso:2002,Cioranescu:Damlamian:Griso:2008}.
Part of our results are consequences of some properties already proven in \cite{AAB2017},
but the novelty of the present research relies on the new structure of
the equation under consideration, which cannot be reduced to the classical Fick case
considered in \cite{AAB2017}. Moreover, the presence of the
non--product perturbation
in the Fokker coefficient
represents a further non trivial feature of the problem.

It is worthwhile also to point out that the resulting homogenized
equation has a non--standard
structure, since it remains in an integral form with respect to the microvariables
and, moreover, the capacity, the diffusivity and the
Fokker coefficients mix
in the limit (see Subsection \ref{s:spec}),
similarly as in the pure Fick case studied in \cite{AAB2017},
where the capacity and the diffusion coefficients
appear in a mixed form in the upscaled equation (see \eqref{b=2a}).

Only when the diffusion matrix is the identity (i.e., in the pure
Fokker--Planck case) and the
capacity is constant, the limit equations
assume the standard form analogous to the
starting one (see equation \eqref{nppr240}),
and in this case the memory of the periodic microstructure
remains in the limit only as an average of the coefficients.

\medskip

The paper is organized as follows: in Section \ref{s:int050},
we present the general problem.
In Section~\ref{s:p-mod} we introduce the weaker version
that
we shall be able to treat rigorously, state
our main results,
and discuss some heuristics based on formal expansions.
Some preliminary statements are proven in Section~\ref{s:pre}
and, finally, in Section~\ref{s:hom} we state and prove our
rigorous results.

\section{The problem}
\label{s:FPpro}
\par\noindent
Let $\Omega$ be an open connected bounded set in $\mathbb{R}^n$
with Lipschitz boundary, $T>0$,
and set $\Omega_T=\Omega\times(0,T)$.
Let
$\mathcal{Y}=(0,1)^n$,
$\mathcal{S}=(0,1)$,
and call
$\mathcal{Q}=\mathcal{Y}\times\mathcal{S}$
the \emph{microscopic cell} or, simply, the \emph{cell}.

Given a function $w\in L^2(\Omega)$ (or $w\in L^2(\Omega_T)$), we will denote
by $\Vert w\Vert_2$ its $L^2(\Omega)$-norm (or $L^2(\Omega_T)$-norm, respectively).
Finally, $\const$ will denote a strictly positive constant, which may vary from line to line.

\subsection{The general problem}
\label{s:int050}
\par\noindent
Consider the real functions
$a(x,t,y,\tau)$,
$c(x,t,y,\tau)$,
and
the $n\times n$--matrix function
$B(x,t,y,\tau)$
with $(x,t)\in\Omega_T$ and
$\mathcal{Q}$--periodic in $(y,\tau)$.
We assume that
$B\in
L^\infty(\Omega_T\times\mathcal{Q};\mathbb{R}^{n\times n})$
is symmetric and
satisfies the bounds
\begin{equation}
\label{ppro00-000}
C^{-1}|\xi|^2\le B_{ij}(x,t,y,\tau)\xi_i\xi_j\le C|\xi|^2,
\end{equation}
for every $\xi\in\mathbb{R}^n$ and almost every
$(x,t,y,\tau)\in\Omega_T\times\mathcal{Q}$.
We assume, also, that
$a,c\in
L^\infty(\Omega_T\times\mathcal{Q})$
satisfy the bounds
\begin{equation}
\label{ppro00-200}
C^{-1}\le a(x,t,y,\tau),c(x,t,y,\tau)\le C\,,
\end{equation}
for almost every $(x,t,y,\tau)\in\Omega_T\times\mathcal{Q}$.
Moreover, we assume that
$a,c,B_{ij}$ are Lipschitz--continuous on $\overline\Omega_T\times\overline {\mathcal{Q}}$.

Let $\alpha,\beta>0$
and set
\begin{equation}
\label{gpro000}
a^\varepsilon(x,t)=a\Big(x,t,\frac{x}{\varepsilon^\alpha},
\frac{t}{\varepsilon^\beta}\Big),\quad
c^\varepsilon(x,t)=c\Big(x,t,\frac{x}{\varepsilon^\alpha},
\frac{t}{\varepsilon^\beta}\Big),\quad
B^\varepsilon(x,t)
=
B\Big(x,t,\frac{x}{\varepsilon^\alpha},\frac{t}{\varepsilon^\beta}\Big)
.
\end{equation}
Given $f\in L^2(\Omega_T)$
and $\bar u\in H^1_0(\Omega)$,
we are interested in studying the family of mixed Fick and Fokker--Planck problems
with oscillating coefficients
\begin{align}
\label{gpro010}
a^\varepsilon
\frac{\partial u_\varepsilon}{\partial t}
-\Div
(
 B^\varepsilon\nabla(c^\varepsilon u_\varepsilon)
)
=f
,
&
\;\;\;
\textrm{ in }
\Omega_T=\Omega\times(0,T);
\\
\label{gpro020}
u_\varepsilon(x,t)=0
,
&
\;\;\;
\textrm{ on }
\partial\Omega\times(0,T);
\\
\label{gpro050}
u_\varepsilon(x,0)=
\bar u(x)
,
&
\;\;\;
\textrm{ in }
\Omega
.
\end{align}
Note that, in the case $c=1$, the pure Fick problem is recovered,
while in the case $B=I$, we obtain the pure Fokker-Planck problem.
The terms $a$, $B$, $c$, and $f$ will be respectively called
\textit{capacity coefficient}, \textit{diffusion matrix},
\textit{Fokker coefficient},
and
\textit{source term}.

If we let $v_\varepsilon=c^\varepsilon u_\varepsilon$, the above problem
can be rewritten as the following Fick problem with linear
lower order terms
\begin{align}
\label{gpro060}
\frac{\partial v_\varepsilon}{\partial t}
-\Div
\Big(
 \frac{c^\varepsilon}{a^\varepsilon}
 B^\varepsilon\nabla v_\varepsilon
\Big)
+
 B^\varepsilon
\nabla
\Big(
 \frac{c^\varepsilon}{a^\varepsilon}
\Big)
\cdot
 \nabla v_\varepsilon
-
\frac{1}{c^\varepsilon}
\frac{\partial c^\varepsilon}{\partial t}
v_\varepsilon
=
\frac{c^\varepsilon}{a^\varepsilon}
f
,
&
\;\;\;
\textrm{ in }
\Omega_T=\Omega\times(0,T);
\\
\label{gpro070}
v_\varepsilon(x,t)=0
,
&
\;\;\;
\textrm{ on }
\partial\Omega\times(0,T);
\\
\label{gpro080}
u_\varepsilon(x,0)=
c^\varepsilon(x,0)
\bar u(x)
,
&
\;\;\;
\textrm{ in }
\Omega
.
\end{align}

We remark that, by \cite[Chapter~4, Theorem~9.1]{LSU},
for every $\varepsilon>0$ fixed,
the problem \eqref{gpro060}--\eqref{gpro080}
admits a unique solution
$v_\varepsilon\in L^2(0,T;H^2(\Omega))\cap H^1(0,T;L^2(\Omega))$.
Clearly this implies existence and uniqueness of the solution
$u_\varepsilon\in L^2(0,T;H^1(\Omega))\cap H^1(0,T;L^2(\Omega))$
of the problem \eqref{gpro010}--\eqref{gpro050}.

As we pointed out above, the homogenization of the previous problem
in its full generality provides some very hard technical difficulties.
For this reason, we shall treat only the special \emph{weakly non--product}
case described in the following Section~\ref{s:p-mod}.

\subsection{The weakly non--product problem}
\label{s:p-mod}
\par\noindent
Here we consider a special case in which the coefficients
$a$ and $c$
of the problem \eqref{gpro010}--\eqref{gpro050}
are factored in one term depending on $(x,t,y)$ and
another on $(x,t,\tau)$; namely, the dependence on the
micro--variables is separated.
However, for the coefficient $c$, typical of the Fokker--Planck
equation, we can admit a small general perturbation of the
product part.

Consider the real functions
$a_1(x,t,y)$,
$a_2(x,t,\tau)$,
$b_1(x,t,y)$,
$b_2(x,t,\tau)$,
and
$b(x,t,y,\tau)$,
with $(x,t)\in\Omega_T$ and
$\mathcal{Q}$--periodic in $(y,\tau)$.
We assume, also, that
$a_1,b_1\in
L^\infty(\Omega_T\times\mathcal{Y})$,
$a_2,b_2\in
L^\infty(\Omega_T\times\mathcal{S})$,
and
$b\in
L^\infty(\Omega_T\times\mathcal{Q})$
satisfy the bounds
\begin{equation}
\label{ppro00-500}
C^{-1}\le
a_1(x,t,y),
a_2(x,t,\tau),
b_1(x,t,y),
b_2(x,t,\tau),
b(x,t,y,\tau)
\le C\,,
\end{equation}
for almost every $(x,t,y,\tau)\in\Omega_T\times\mathcal{Q}$.
Moreover, we assume that
$a_1,b_1$ are Lipschitz--continuous in $\overline\Omega_T\times\overline{\mathcal{Y}}$,
$a_2,b_2$ in $\overline\Omega_T\times\overline{\mathcal{S}}$,
and
$b$ in $\overline\Omega_T\times\overline{\mathcal{Q}}$.
We keep the same assumptions as in Subsection~\ref{s:int050}
for $B$, $f$, and $\bar u$.

Similarly as above, for $(x,t)\in\Omega_T$, we set
\begin{equation}
\label{prob000}
a_1^\varepsilon(x,t)=a_1\Big(x,t,\frac{x}{\varepsilon^\alpha}\Big),  \;\;
a_2^\varepsilon(x,t)=a_2\Big(x,t,\frac{t}{\varepsilon^\beta}\Big),\;\;
\end{equation}
\begin{equation}
\label{prob010}
b_1^\varepsilon(x,t)=b_1\Big(x,t,\frac{x}{\varepsilon^\alpha}\Big),  \;\;
b_2^\varepsilon(x,t)=b_2\Big(x,t,\frac{t}{\varepsilon^\beta}\Big), \;\;
\end{equation}
and
\begin{equation}
\label{ppro000}
b^\varepsilon(x,t)=b\Big(x,t,\frac{x}{\varepsilon^\alpha},
\frac{t}{\varepsilon^\beta}\Big).
\end{equation}
For $\varepsilon>0$,
we will study the family of problems
\begin{align}
\label{prob030}
a_1^\varepsilon
a_2^\varepsilon
\frac{\partial u_\varepsilon}{\partial t}
-\Div
(
 B^\varepsilon\nabla(
 (b_1^\varepsilon b_2^\varepsilon+\varepsilon b^\varepsilon)
 u_\varepsilon)
)
=f
,
&
\;\;\;
\textrm{ in }
\Omega_T;
\\
\label{prob040}
u_\varepsilon(x,t)=0
,
&
\;\;\;
\textrm{ on }
\partial\Omega\times(0,T);
\\
\label{prob050}
u_\varepsilon(x,0)=
\bar u(x)
,
&
\;\;\;
\textrm{ in }
\Omega
,
\end{align}
where $\varepsilon>0$.
Note that in the case $b=0$ and $b_1=b_2=1$, we recover
the pure Fick problem
discussed in \cite[Section~3.2]{AAB2017}.
We shall also consider the
following auxiliary problem for
the function $v_\varepsilon(x,t)=b_1^\varepsilon(x,t)u_\varepsilon(x,t)$:
\begin{align}
\label{fas000}
a_1^\varepsilon
a_2^\varepsilon
\frac{\partial}{\partial t}\Big(\frac{v_\varepsilon}{b_1^\varepsilon}\Big)
-\Div
\Big(
 B^\varepsilon\nabla\Big(
 \Big(b_2^\varepsilon+\varepsilon\frac{b^\varepsilon}{b_1^\varepsilon}\Big)
 v_\varepsilon\Big)
\Big)
=f
,
&
\;\;\;
\textrm{ in }
\Omega_T;
\\
\label{fas010}
v_\varepsilon(x,t)=0
,
&
\;\;\;
\textrm{ on }
\partial\Omega\times(0,T);
\\
\label{fas020}
v_\varepsilon(x,0)=
\bar v_\varepsilon(x)
,
&
\;\;\;
\textrm{ in }
\Omega
,
\end{align}
where $\bar v_\varepsilon(x)=b_1^\varepsilon(x,0)\bar u(x)$.
Note that, since
$b_1(\cdot,0,\cdot)\in L^\infty(\Omega\times\mathcal{Y})$,
it follows
\begin{equation}
\label{fas022}
\|\bar{v}_\varepsilon\|_2
\le
C
\,.
\end{equation}
On the other hand $\|\nabla\bar{v}_\varepsilon\|_2=O(\varepsilon^{-\alpha}$).

Note that the weak formulation of problem \eqref{fas000}--\eqref{fas020}
can be written as follows:
\begin{multline}
\label{fas023}
-\int_0^T\int_\Omega
\frac{v_\varepsilon}{b^\varepsilon_1}
\frac{\partial}{\partial t}(a^\varepsilon_1 a^\varepsilon_2\phi)
\,\textrm{d}x\,\textrm{d}t
+
\int_0^T\int_\Omega
 B^\varepsilon\nabla\Big(
 \Big(b_2^\varepsilon+\varepsilon\frac{b^\varepsilon}{b_1^\varepsilon}\Big)
 v_\varepsilon\Big)
\cdot
\nabla\phi
\,\textrm{d}x\,\textrm{d}t
\\
=
\int_0^T\int_\Omega
f\phi
\,\textrm{d}x\,\textrm{d}t
+
\int_\Omega
\frac{\bar{v}_\varepsilon(x)}{b^\varepsilon_1(x,0)}
a^\varepsilon_1(x,0)
a^\varepsilon_2(x,0)
\phi(x,0)
\,\textrm{d}x\,,
\end{multline}
for any $\phi\in H^1(\Omega_T)$ such that $\phi=0$ on
$\partial\Omega\times(0,T)$ and
$\phi(x,T)=0$ a.e.\ in $\Omega$.

For later use, we set
\begin{equation*}
  \omega_{\alpha,1}=
  \begin{cases}
    1\,,&\qquad \alpha=1\,,
    \\
    0\,,&\qquad 0<\alpha<1\,.
  \end{cases}
\end{equation*}

The main result of the paper is the following homogenization theorem, whose proof can be found in Subsection \ref{s:main_dim}.

\begin{theorem}
\label{t:b=2a}
Assume $0<\alpha\leq 1$ and $\beta >0$. For any $\eps>0$, let $u_\varepsilon$ be the unique solution of problem \eqref{prob030}--\eqref{prob050}.
Then, when $\eps\to 0$, $u_\eps\rightharpoonup u$ weakly in $L^2(\Om_T)$, where $u\in L^2(0,T;H^1_0(\Omega))$ is the unique weak solution
of the homogenized problem
\begin{multline}
\int_{\mathcal{Q}}
\Big[
\frac{a_1}{\int_{\mathcal{S}} a_2^{-1}\,\emph{d}\tau}
\frac{\partial}{\partial t}
\Big(
     \frac{u}{b_1\int_\mathcal{Y}b_1^{-1}\,\emph{d}y}
\Big)
-
\frac{1}{a_2\int_{\mathcal{S}} a_2^{-1}\,\emph{d}\tau}
\emph{div}
\Big(
B_{\emph{eff}}
\nabla\Big(b_2
\frac{u}{\int_\mathcal{Y} b_1^{-1}\,\emph{d}y}\Big)
\Big)
\\
-
\frac{1}{a_2\int_{\mathcal{S}} a_2^{-1}\,\emph{d}\tau}
\emph{div}
\Big(
B
\nabla_y\Big(
             \omega_{\alpha,1}\frac{b}{b_1}-b_2\zeta
             +\chi\cdot\nabla b_2
        \Big)
    \frac{u}{\int_\mathcal{Y}b_1^{-1}\,\emph{d}y}
\Big)
\Big]
\,\emph{d}y\,\emph{d}\tau
=
f\,,
\quad
\emph{ in } \Omega_T\,;
\label{b=2a}
\end{multline}
and
\begin{equation}
\label{b=2aini}
\frac{u(x,0)}{\int_{\mathcal{Y}} b_1^{-1}\,\emph{d}y}
=\bar u(x)\left(\int_{\mathcal{Y}}a_1(x,0,y)\,\emph{d}y\right)
       \Big(
            \int_{\mathcal{Y}}
              \frac{a_1(x,0,y)}{b_1(x,0,y)}
              \,\emph{d}y
       \Big)^{-1}\,,
\quad
\emph{ in } \Omega
\,.
\end{equation}
Here, the matrix $B_{\emph{eff}}$ is given by
\begin{equation}
\label{b=2a000}
B_{\emph{eff}}
=
B
\nabla_y(y-\chi)
\end{equation}
and
\\
1) if $\beta= 2\alpha$, $\chi$ and $\zeta$ are the solutions of
\eqref{fas213} and \eqref{fas215}, respectively;
\\
2) if $\beta>2\alpha$, $\chi$ and $\zeta$ are the solutions of
\eqref{fas>213} and \eqref{fas>215}, respectively;
\\
3) if $\beta<2\alpha$, $\chi$ and $\zeta$ are the solutions of
\eqref{fas<213} and \eqref{fas<215}, respectively.
\end{theorem}

\subsubsection{Formal expansions for the weakly non--product problem}
\label{s:np-mod}
\par\noindent
In Section~\ref{s:hom},
we will prove rigorously the macroscopic equations for
problem \eqref{prob030}--\eqref{prob050} in the
case $0<\alpha\le1$ and $\beta>0$.
Namely, we will be able to homogenize the system
in the case in which the spatial oscillations are not too
fast with respect to the amplitude of the non--product
perturbation in the Fokker coefficient
appearing in the Fokker--Planck equation \eqref{prob030}.

However, before the rigorous approach, we first set up some formal expansions. More precisely,
in Subsection~\ref{s:np-mod01}, we consider, as an example,
the case $\alpha=1$ and $\beta=2$, which is indeed rigorously covered
by Theorem~\ref{t:b=2a}. Moreover,
in Subsection~\ref{s:np-mod02}
we formally approach also some cases (with integer exponents) not covered
by the theory developed in Section~\ref{s:hom}, that is to say,
when the spatial oscillations are faster than the amplitude of the
non--product perturbation in the Fokker coefficient (i.e., $\alpha>1$).
In details, we consider the case $\alpha=2$
and $\beta=1,2,4$, in which
time oscillations are respectively slower, as fast as, and
faster
than spatial ones.
Note that the case
$\alpha=2$ and $\beta=4$ corresponds to the natural parabolic scaling.
In our formal expansions arguments, we
assume that the diffusion matrix is $B=\text{Id}$,
the capacity coefficients are $a_1=a_2=1$, and the source term is $f=0$.
Note that this last assumption could be easily removed.

\subsubsection{Space oscillations as fast as the perturbation amplitude}
\label{s:np-mod01}
\par\noindent
As mentioned above in this section, we formally study the case
$\alpha=1$ and $\beta=2$, which is covered by Theorem~\ref{t:b=2a}.

We let $y=x/\varepsilon$
and
$\tau=t/\varepsilon^2$ and,
by abusing the notation, we write the differential rules
\begin{equation}
\label{nppr040}
\frac{\partial}{\partial t}=\frac{\partial}{\partial t}
+\frac{1}{\varepsilon^2}\frac{\partial}{\partial\tau},
\;\;
\nabla_x=\nabla_x+\frac{1}{\varepsilon}\nabla_y,
\textrm{ and }\;\;
\Delta_x=\Delta_x
+
\frac{1}{\varepsilon}
(
 \nabla_y\cdot\nabla_x
 +\nabla_x\cdot\nabla_y
)
+
\frac{1}{\varepsilon^2}\Delta_y
.
\end{equation}
We then look for a solution of \eqref{prob030}, using the formal expansion
\begin{equation}
\label{nppr050}
u_\varepsilon(x,t)
=
u_0(x,t,y,\tau)
+
\varepsilon
u_1(x,t,y,\tau)
+
\varepsilon^2
u_2(x,t,y,\tau)
+\cdots
,
\end{equation}
with $u_k$ a $\mathcal{Q}$-periodic function with respect to $(y,\tau)$.
By replacing \eqref{nppr050} in \eqref{prob030}, we get
\begin{equation}
\label{nppr060}
\frac{\partial u_\varepsilon}{\partial t}
=
\frac{\partial u_0}{\partial t}
+
\frac{1}{\varepsilon^2}
\frac{\partial u_0}{\partial \tau}
+
\varepsilon
\frac{\partial u_1}{\partial t}
+
\frac{1}{\varepsilon}
\frac{\partial u_1}{\partial \tau}
+
\varepsilon^2
\frac{\partial u_2}{\partial t}
+
\frac{\partial u_2}{\partial \tau}
+
\varepsilon
\frac{\partial u_3}{\partial \tau}
+
\varepsilon^2
\frac{\partial u_4}{\partial \tau}
+o(\varepsilon^2)
\end{equation}
and
\begin{align}
\label{nppr070}
\Delta  [(b_1^\varepsilon b_2^\varepsilon+\varepsilon b^\varepsilon)u_\varepsilon]
=
&
\;
\Delta_x
  [(b_1 b_2+\varepsilon b)u_0]
+
\frac{1}{\varepsilon}
 \nabla_y\cdot\nabla_x
  [(b_1 b_2+\varepsilon b)u_0]
\nonumber \\
&
+
\frac{1}{\varepsilon}
 \nabla_x\cdot\nabla_y
  [(b_1 b_2+\varepsilon b)u_0]
+
\frac{1}{\varepsilon^2}
\Delta_y
  [(b_1 b_2+\varepsilon b)u_0]
\nonumber \\
&
+
\varepsilon
\Delta_x
  [(b_1 b_2+\varepsilon b)u_1]
+
 \nabla_y\cdot\nabla_x
  [(b_1 b_2+\varepsilon b)u_1]
\nonumber \\
&
+
 \nabla_x\cdot\nabla_y
  [(b_1 b_2+\varepsilon b)u_1]
+
\frac{1}{\varepsilon}
\Delta_y
  [(b_1 b_2+\varepsilon b)u_1]
\nonumber\\
&
+
\varepsilon^2
\Delta_x
  [(b_1 b_2+\varepsilon b)u_2]
+
\varepsilon
 \nabla_y\cdot\nabla_x
  [(b_1 b_2+\varepsilon b)u_2]
\nonumber \\
&
+
\varepsilon
 \nabla_x\cdot\nabla_y
  [(b_1 b_2+\varepsilon b)u_2]
+
\Delta_y
  [(b_1 b_2+\varepsilon b)u_2]
\nonumber\\
&
+
\varepsilon^2
 \nabla_y\cdot\nabla_x
  [b_1 b_2 u_3]
+
\varepsilon^2
 \nabla_x\cdot\nabla_y
  [b_1 b_2 u_3]
\nonumber \\
&
+
\varepsilon
\Delta_y
  [(b_1 b_2+\varepsilon b)u_3]
+
\varepsilon^2
\Delta_y
  [b_1 b_2 u_4]
+o(\varepsilon^2)
\,.
\end{align}

Thus, at order $1/\varepsilon^2$, we find the equation
\begin{equation}
\label{nppr080}
\frac{\partial u_0}{\partial \tau}
-
\Delta_y (b_1 b_2 u_0)
=0\,.
\end{equation}
Recalling that $b_1$ does not depend on
$\tau$, see \eqref{prob010},
we let
$v_0(x,t,y,\tau)=b_1(x,t,y) u_0(x,t,y,\tau)$
and
find for $v_0$ the equation
\begin{equation}
\label{nppr090}
\frac{1}{b_1}\frac{\partial v_0}{\partial \tau}
-
\Delta_y (b_2 v_0)
=0\,, 
\end{equation}
which must be solved assuming that $v_0$ is $\mathcal{Q}$-periodic in $(y,\tau)$.
We prove, indeed, that $v_0$ does not depend on the microscopic
variables: we first multiply \eqref{nppr090} times $v_0$ and
integrate on the microscopic cell
\begin{align}
\label{nppr100}
0
&
=
\int_\mathcal{Q}
\frac{1}{b_1}\frac{\partial v_0}{\partial \tau}
v_0
\,\text{d}y\,\text{d}\tau
-
\int_\mathcal{Q}
[\Delta_y (b_2 v_0)]
v_0
\,\text{d}y\,\text{d}\tau
\nonumber\\
&
=
\int_\mathcal{Y}
\frac{1}{2b_1}
\left(\int_\mathcal{S}
\frac{\partial v_0^2}{\partial \tau}
\,\text{d}\tau\right)\text{d}y
+
\int_\mathcal{Q}\!\!
b_2
\nabla_y v_0\cdot\nabla_yv_0
\,\text{d}y\,\text{d}\tau
-
\!\!\int_\mathcal{S}\!\!
b_2
\left(\int_{\partial \mathcal{Y}}
v_0\nabla_y v_0\cdot\nu
\,\text{d}\sigma\right)\text{d}\tau\,.
\end{align}
By periodicity, the first and the third integral vanish, hence
\begin{equation}
\label{nppr110}
0
=
\int_\mathcal{Q}
b_2
|\nabla_y v_0|^2
\,\text{d}y\,\text{d}\tau
\ge
C^{-1}
\int_\mathcal{Q}
|\nabla_y v_0|^2
\,\text{d}y\,\text{d}\tau\,,
\end{equation}
where we used \eqref{ppro00-500}.
This implies that $v_0$ is constant with respect to $y$.

On the other hand, since both $b_2$ and $v_0$ do not depend on $y$,
from \eqref{nppr090} we immediately get that $v_0$ does not depend on
$\tau$ as well.
Note that, since $v_0(x,t)=b_1(x,t,y)u_0(x,t,y,\tau)$, we have that
$u_0$ does not depend on $\tau$.

We now consider the $1/\varepsilon$ order equation.
From
\eqref{prob030}, \eqref{nppr060}, and \eqref{nppr070}, we have
\begin{equation}
\label{nppr120}
\frac{\partial u_1}{\partial \tau}
-
\nabla_y\cdot\nabla_x (b_1 b_2 u_0)
-
\nabla_x\cdot\nabla_y (b_1 b_2 u_0)
-
\Delta_y (bu_0)
-
\Delta_y (b_1 b_2 u_1)
=0
.
\end{equation}
Since $b_2$ and $v_0=b_1u_0$ do not depend on $y$, \eqref{nppr120}
simplifies to
\begin{equation}
\label{nppr130}
\frac{\partial u_1}{\partial \tau}
-
\Delta_y (bu_0)
-
\Delta_y (b_1 b_2 u_1)
=0
.
\end{equation}
We now let
$v_1(x,t,y,\tau)=b_1(x,t,y)u_1(x,t,y,\tau)$ and, from \eqref{nppr130},
we get
\begin{equation}
\label{nppr140}
\frac{1}{b_1}
\frac{\partial v_1}{\partial \tau}
-
b_2
\Delta_y v_1
=
\Delta_y
\Big(
     \frac{b}{b_1}
\Big)
v_0
.
\end{equation}
We now look for a solution of the above equation in the factored
form
\begin{equation}
\label{nppr150}
v_1(x,t,y,\tau)=-\zeta(x,t,y,\tau)v_0(x,t)\,,
\end{equation}
with $\zeta$ a $\mathcal{Q}$-periodic function with respect to $(y,\tau)$.
By plugging \eqref{nppr150} into \eqref{nppr140}, we get
that $\zeta$ has to solve the equation
\begin{equation}
\label{nppr160}
\frac{1}{b_1}
\frac{\partial \zeta}{\partial \tau}
-
b_2
\Delta_y \zeta
=
-
\Delta_y
     \Big(
     \frac{b}{b_1}
     \Big)
.
\end{equation}

We, finally, consider the $\varepsilon^0$ order equation, which will yield
a compatibility condition providing an equation for $u_0$.
From
\eqref{prob030}, \eqref{nppr060}, and \eqref{nppr070}, we have
\begin{align}
\label{nppr170}
\frac{\partial u_0}{\partial t}
+
\frac{\partial u_2}{\partial \tau}
&
-
\Delta_x(b_1b_2u_0)
-
\nabla_y\cdot\nabla_x(bu_0)
-
\nabla_x\cdot\nabla_y(bu_0)
-
\nabla_y\cdot\nabla_x(b_1 b_2 u_1)
\nonumber\\
&
-
\nabla_x\cdot\nabla_y(b_1 b_2 u_1)
-
\Delta_y(bu_1)
-
\Delta_y(b_1 b_2 u_2)
=0
,
\end{align}
which can be seen as an equation for $u_2$. Hence, as usual,
we introduce the function
$v_2(x,t,y,\tau)=b_1(x,t,y)u_2(x,t,y,\tau)$ and rewrite \eqref{nppr170}
as
\begin{align}
\label{nppr180}
\frac{1}{b_1}
\frac{\partial v_2}{\partial \tau}
-
b_2
\Delta_y v_2
=
&
\Delta_x(b_2 v_0)
+
\nabla_y\cdot\nabla_x(bu_0)
+
\nabla_x\cdot\nabla_y(bu_0)
+
\nabla_y\cdot\nabla_x(b_2 v_1)
\nonumber\\
&
+
\nabla_x\cdot\nabla_y(b_2 v_1)
+
\Delta_y(bu_1)
-
\frac{\partial u_0}{\partial t}\,,
\end{align}
for $v_2$ a $\mathcal{Q}$-periodic function with respect to $(y,\tau)$.

Now, if we integrate \eqref{nppr180} on $\mathcal{Q}$, since $b_1$
does not depend on $\tau$ and $b_2$ does not depend on $y$, on the
left hand side we find zero. Hence, we have the compatibility condition
\begin{align}
\label{nppr190}
&
\int_\mathcal{Q}
\Big[\frac{\partial u_0}{\partial t}
-\Big(
\Delta_x(b_2 v_0)
+
\nabla_y\cdot\nabla_x(bu_0)
+
\nabla_x\cdot\nabla_y(bu_0)
+
\nabla_y\cdot\nabla_x(b_2 v_1)
\nonumber\\
&
\phantom{
         \int_\mathcal{Q} \Big[
        }
+
\nabla_x\cdot\nabla_y(b_2 v_1)
+
\Delta_y(bu_1)
\Big)\Big]
\,\text{d}y\,\text{d}\tau
=0\,.
\end{align}
By the periodicity on $\mathcal{Q}$ and Gauss-Green formulas, we also have
\begin{equation}
\label{nppr200}
\int_\mathcal{Q}
\Big[
\frac{\partial u_0}{\partial t}
-\Big(\Delta_x(b_2 v_0)
+
\nabla_x\cdot\nabla_y(bu_0)
+
\nabla_x\cdot\nabla_y(b_2 v_1)
\Big)\Big]
\,\text{d}y\,\text{d}\tau
=0\,.
\end{equation}
Since, again by periodicity,
\begin{multline}
\label{nppr210}
\int_\mathcal{Q}
\Big[
\nabla_x\cdot\nabla_y(bu_0)
+
\nabla_x\cdot\nabla_y(b_2 v_1)
\Big]
\,\text{d}y\,\text{d}\tau
=
\\
\int_\mathcal{S}
\Big[
\nabla_x\cdot
\int_\mathcal{Y}
\nabla_y(bu_0)
\,\text{d}y
+
\nabla_x\cdot
\int_\mathcal{Y}
\nabla_y(b_2 v_1)
\,\text{d}y
\Big]
\,\text{d}\tau
=0
,
\end{multline}
we have
\begin{equation}
\label{nppr220}
\frac{\partial }{\partial t}
\Big[
\Big(
\int_\mathcal{Y}
\frac{1}{b_1}
\,\text{d}y
\Big)
v_0
\Big]
-
\Delta_x
\Big[
\Big(
\int_\mathcal{S}
b_2
\,\text{d}\tau
\Big)
v_0
\Big]
=0,
\end{equation}
where we have used that $v_0$ does not depend on $y$ and $\tau$.

We have, finally, found an equation for $v_0$. Indeed, we can deduce
the equation that must be satisfied by the mean value of $u_0$ on the
microscopic cell.
If we let
\begin{equation}
\label{nppr230}
u(x,t)
=
\int_\mathcal{Q}u_0(x,t,y,\tau)
\,\text{d}y\,\text{d}\tau
=
\int_\mathcal{Q}\frac{v_0(x,t)}{b_1(x,t,y)}
\,\text{d}y\,\text{d}\tau
=
v_0(x,t)\int_\mathcal{Y}\frac{\text{d}y}{b_1(x,t,y)}
,
\end{equation}
we can rewrite \eqref{nppr220} as an equation for $u$,
finding
\begin{equation}
\label{nppr240}
\frac{\partial u}{\partial t}
-
\Delta_x
\Big[
\Big(
\int_\mathcal{S}
b_2
\,\text{d}\tau
\Big)
\Big(
\int_\mathcal{Y}
\frac{\text{d}y}{b_1}
\Big)^{-1}
u
\Big]
=0\,.
\end{equation}
It is interesting to note that the non--product small
correction $\varepsilon b^\varepsilon$ in \eqref{prob030}
does not play any role in the upscaled equation.

Note that equation \eqref{nppr220} (resp.\ \eqref{nppr240})
coincides with the rigorous equations obtained in \eqref{fas200} (resp.\ \eqref{b=2a}),
once we have taken into account that,
under the present assumptions, we have:
i) the cell functions
$\chi^j$ in Theorem~\ref{t:b=2afact} are identically
equal to zero;
ii) the cell function
$\zeta$ in Theorem~\ref{t:b=2afact}, which is equal to
the function
$\zeta$ introduced in the equation \eqref{nppr160} above,
and the term
$\nabla_y(b/b_1)$
disappear
from the expressions of $P_\text{eff}$ in \eqref{fas220} and
$z_\text{eff}$ in \eqref{fas230}, due to the periodicity.

\subsubsection{Space oscillations faster than the perturbation amplitude}
\label{s:np-mod02}
\par\noindent
In this section, we formally study the homogenization for the
equation \eqref{prob030} in some cases not covered by the rigorous theory
developed in Section~\ref{s:hom}.
As mentioned above, we shall consider situations in which the
spatial oscillation is faster than the amplitude of the
non--product perturbation present in the Fokker coefficient.

We remark that, as we shall prove in Section~\ref{s:hom}
(for $\alpha<1$ or $\alpha=1$ and $B=I$)
and as we found in Subsection~\ref{s:np-mod01} (for $\alpha=1$),
the non--product perturbation $\varepsilon b^\varepsilon$ appearing
in \eqref{prob030} does not affect the upscaled equation.
In the three cases discussed below, we shall see that this property
is preserved in the case $\alpha=2$, namely, even when the spatial
oscillation is fast, which is expected to reinforce the effect of the perturbation.
We cannot conclude that this is a general result for the scaling
$\alpha>1$;
indeed, it might depend on our peculiar
choice of the diffusion matrix and the capacity coefficients
in the
formal computation.

We first consider the problem \eqref{prob030}--\eqref{prob050}
for $\alpha=\beta=2$.
Indeed, from the point of view of computations,
such a case seems to be the most delicate among those discussed in this section.
We then let $y=x/\varepsilon^2$
and
$\tau=t/\varepsilon^2$ and,
by abusing the notation, we write the differential rules
\begin{equation}
\label{nppr340}
\frac{\partial}{\partial t}=\frac{\partial}{\partial t}
+\frac{1}{\varepsilon^2}\frac{\partial}{\partial\tau},
\;\;
\nabla_x=\nabla_x+\frac{1}{\varepsilon^2}\nabla_y,
\textrm{ and }\;\;
\Delta_x=\Delta_x
+
\frac{1}{\varepsilon^2}
(
 \nabla_y\cdot\nabla_x
 +\nabla_x\cdot\nabla_y
)
+
\frac{1}{\varepsilon^4}\Delta_y
\end{equation}
and look for a solution of \eqref{prob030}, using the formal expansion \eqref{nppr050}. Differentiating in time, we are led again to
\eqref{nppr060}, while differentiation in space yields
\begin{align}
\label{nppr370}
\nabla\cdot
 \nabla
  [(b_1^\varepsilon b_2^\varepsilon+\varepsilon b^\varepsilon)u_\varepsilon]
=
&
\;
\Delta_x
  [b_1 b_2 u_0]
+
\frac{1}{\varepsilon^2}
 \nabla_y\cdot\nabla_x
  [(b_1 b_2+\varepsilon b)u_0]
\nonumber \\
&
+
\frac{1}{\varepsilon^2}
 \nabla_x\cdot\nabla_y
  [(b_1 b_2+\varepsilon b)u_0]
+
\frac{1}{\varepsilon^4}
\Delta_y
  [(b_1 b_2+\varepsilon b)u_0]
\nonumber \\
&
+
\frac{1}{\varepsilon}
 \nabla_y\cdot\nabla_x
  [(b_1 b_2+\varepsilon b)u_1]
+
\frac{1}{\varepsilon}
 \nabla_x\cdot\nabla_y
  [(b_1 b_2+\varepsilon b)u_1]
\nonumber \\
&
+
\frac{1}{\varepsilon^3}
\Delta_y
  [(b_1 b_2+\varepsilon b)u_1]
+
 \nabla_y\cdot\nabla_x
  [b_1 b_2 u_2]
\nonumber \\
&
+
 \nabla_x\cdot\nabla_y
  [b_1 b_2 u_2]
+
\frac{1}{\varepsilon^2}
\Delta_y
  [(b_1 b_2+\varepsilon b)u_2]
\nonumber\\
&
+
\frac{1}{\varepsilon}
\Delta_y
  [(b_1 b_2+\varepsilon b)u_3]
+
\Delta_y
  [b_1 b_2 u_4]
+o(1),
\end{align}
where we took into account the powers of $\varepsilon$ up to the order
$\varepsilon^0$.

Thus, at order $1/\varepsilon^4$, we find the equation
\begin{equation}
\label{nppr380}
\Delta_y (b_1 b_2 u_0)
=0\,.
\end{equation}
Recalling that $b_2$ does not depend on
$y$ (see \eqref{prob010}),
from \eqref{nppr380}
we have that
$b_1(x,t,y) u_0(x,t,y,\tau)$ does not depend on $y$,
thus we set
$v_0(x,t,\tau)=b_1(x,t,y) u_0(x,t,y,\tau)$.

We now consider the $1/\varepsilon^3$ order equation.
From
\eqref{prob030}, \eqref{nppr060}, and \eqref{nppr370}, we have
\begin{equation}
\label{nppr390}
\Delta_y(b u_0)
+
\Delta_y(b_1 b_2 u_1)
=0\,,
\end{equation}
which, provided we let
$v_1(x,t,y,\tau)=b_1(x,t,y) u_1(x,t,y,\tau)$, can be rewritten
as
\begin{equation}
\label{nppr400}
b_2
\Delta_y(v_1)
=
-
v_0
\Delta_y\Big(
             \frac{b}{b_1}
        \Big)
,
\end{equation}
where we have used that $b_2$ and $v_0$ do not depend on $y$.
We now look for a solution of the above equation in the factored
form
\begin{equation}
\label{nppr410}
v_1(x,t,y,\tau)=\chi_1(x,t,y,\tau)v_0(x,t,\tau)
.
\end{equation}
By plugging \eqref{nppr410} into \eqref{nppr400} and using again
that $v_0$ does not depend on $y$, we get
that $\chi_1$ has to solve the equation
\begin{equation}
\label{nppr420}
b_2
\Delta_y(\chi_1)
=
-
\Delta_y\Big(
             \frac{b}{b_1}
        \Big)
.
\end{equation}
We now consider the $1/\varepsilon^2$ order equation.
From
\eqref{prob030}, \eqref{nppr060}, and \eqref{nppr370}, we have
\begin{equation}
\label{nppr430}
\frac{\partial u_0}{\partial \tau}
-
[
     \nabla_y\cdot\nabla_x
     (b_1 b_2 u_0)
     +
     \nabla_x\cdot\nabla_y
     (b_1 b_2 u_0)
     +
     \Delta_y(b u_1)
     +
     \Delta_y(b_1 b_2 u_2)
]
=0
.
\end{equation}
Since $b_2$ and $v_0=b_1u_0$ do not depend on $y$, we get
\begin{equation}
\label{nppr440}
\frac{1}{b_1}
\frac{\partial v_0}{\partial \tau}
-
[
     \Delta_y\Big(\frac{b}{b_1} v_1\Big)
     +
     \Delta_y(b_2 v_2)
]
=0
,
\end{equation}
where we set
$v_2(x,t,y,\tau)=b_1(x,t,y) u_2(x,t,y,\tau)$.
Since the last two terms above integrate to zero on $\mathcal{Y}$ and
$v_0$ does not depend on $y$, we
have the compatibility condition
\begin{equation}
\label{nppr450}
\Big(\int_\mathcal{Y}\frac{\text{d}y}{b_1}\Big)
\frac{\partial v_0}{\partial \tau}
=0
,
\end{equation}
which implies that $v_0$ does not depend on $\tau$.
Hence, $v_0=v_0(x,t)$ and,
since $v_0=b_1 u_0$,
also $u_0$ does not depend
on $\tau$, namely $u_0=u_0(x,t,y)$.
Inserting, now, \eqref{nppr410} in \eqref{nppr440}, we get
the following equation for $v_2$:
\begin{equation}
\label{nppr460}
\Delta_y(b_2 v_2)
=
-v_0\Delta_y\Big(\frac{b}{b_1} \chi_1\Big)
.
\end{equation}
We will look for a solution of the above equation in the
factored form
\begin{equation}
\label{nppr470}
v_2(x,t,y,\tau)=\chi_2(x,t,y,\tau)v_0(x,t)
.
\end{equation}
This leads to the equation
\begin{equation}
\label{nppr480}
\Delta_y(b_2 \chi_2)
=
-\Delta_y\Big(\frac{b}{b_1} \chi_1\Big)
\end{equation}
for the unknown function $\chi_2$.

Next we consider the $1/\varepsilon$ order equation.
From
\eqref{prob030}, \eqref{nppr060}, and \eqref{nppr370} we have
\begin{equation}
\label{nppr490}
\frac{\partial u_1}{\partial \tau}
-
[
     (\nabla_x\cdot\nabla_y
     +
     \nabla_y\cdot\nabla_x)
     (b u_0+b_1 b_2 u_1)
     +
     \Delta_y(b u_2)
     +
     \Delta_y(b_1 b_2 u_3)
]
=0
.
\end{equation}
Since all the terms above but the first one
integrate to zero on $\mathcal{Y}$,
we have the compatibility condition
\begin{equation}
\label{nppr500}
\int_\mathcal{Y}
\frac{\partial u_1}{\partial \tau}
\,\text{d}y
=0
,
\end{equation}
that, recalling the definition of $v_1$ given below \eqref{nppr390} and
\eqref{nppr410}, yields the condition
\begin{equation}
\label{nppr510}
\int_\mathcal{Y}
\frac{\partial }{\partial \tau}
\Big(
     \frac{\chi_1}{b_1}
\Big)
\,\text{d}y
=0
,
\end{equation}
which completes the definition of $\chi_1$ as solution of the
equation \eqref{nppr420}.
Setting, now,
$v_3(x,t,y,\tau)=b_1(x,t,y) u_3(x,t,y,\tau)$, we can rewrite
\eqref{nppr490} as an equation for $v_3$; indeed, we find
\begin{equation}
\label{nppr520}
\Delta_y(b_2 v_3)
=
\frac{v_0}{b_1}
\frac{\partial \chi_1}{\partial \tau}
-
\Big[
     (\nabla_x\cdot\nabla_y
     +
     \nabla_y\cdot\nabla_x)
     \Big(
          \frac{b}{b_1} v_0+ b_2 \chi_1 v_0
     \Big)
     +
     v_0
     \Delta_y\Big(
                  \frac{b}{b_1} \chi_2
             \Big)
\Big]
.
\end{equation}
Then, we turn to the $\varepsilon^0$ order equation.
From
\eqref{prob030}, \eqref{nppr060}, and \eqref{nppr370}, we have
\begin{equation}
\label{nppr530}
\frac{\partial u_0}{\partial t}
+
\frac{\partial u_2}{\partial \tau}
-
\Big[
     \Delta_x(b_1 b_2 u_0)
     +
     (\nabla_y\cdot\nabla_x
      +
      \nabla_x\cdot\nabla_y)
     (b u_1 + b_1 b_2 u_2)
     +
     \Delta_y(b u_3)
     +
     \Delta_y(b_1 b_2 u_4)
\Big]
=
0\,.
\end{equation}
Since all the terms above but the first three
on the left
integrate to zero on $\mathcal{Y}$,
we have the compatibility condition
\begin{equation}
\label{nppr540}
\frac{\partial }{\partial t}
\Big(
v_0
\int_\mathcal{Y}
\frac{\text{d}y}{b_1}
\Big)
+
\int_\mathcal{Y}
\frac{\partial u_2}{\partial \tau}
\,\text{d}y
-
     \Delta_x(b_2 v_0)
=0\,,
\end{equation}
where we have used that $v_0=b_1 u_0$ and
that $v_0$ and $b_2$ do not depend on $y$.
Finally, by integrating over $\mathcal{S}$,
using the $\mathcal{Q}$-periodicity of $u_2$ in $(y,\tau)$
and the fact that both $b_1$ and $v_0$ do not depend on $\tau$,
we get for $v_0$ the equation
\begin{equation}
\label{nppr550}
\frac{\partial }{\partial t}
\Big[
\Big(
\int_\mathcal{Y}
\frac{\text{d}y}{b_1}
\Big)
v_0
\Big]
-
\Delta_x
\Big[
\Big(
\int_\mathcal{S}
b_2
\,\text{d}\tau
\Big)
v_0
\Big]
=0
,
\end{equation}
which coincides with the equation \eqref{nppr220}, found in Section~\ref{s:np-mod01}.
Hence, also in this case, equation \eqref{nppr240} is still in force.

The second case we consider here is the problem \eqref{prob030}--\eqref{prob050}
for $\alpha=2$ and $\beta=4$.
We then let $y=x/\varepsilon^2$
and
$\tau=t/\varepsilon^4$ and,
by abusing the notation, we write the differential rules
\begin{equation}
\label{nppr840}
\frac{\partial}{\partial t}=\frac{\partial}{\partial t}
+\frac{1}{\varepsilon^4}\frac{\partial}{\partial\tau},
\;\;
\nabla_x=\nabla_x+\frac{1}{\varepsilon^2}\nabla_y,
\textrm{ and }\;\;
\Delta_x=\Delta_x
+
\frac{1}{\varepsilon^2}
(
 \nabla_y\cdot\nabla_x
 +\nabla_x\cdot\nabla_y
)
+
\frac{1}{\varepsilon^4}\Delta_y
\end{equation}
and look for a solution of \eqref{prob030}, using the formal expansion
\eqref{nppr050}.
By substituting \eqref{nppr050} in \eqref{prob030}, we get
\begin{equation}
\label{nppr860}
\frac{\partial u_\varepsilon}{\partial t}
=
\frac{\partial u_0}{\partial t}
+
\frac{1}{\varepsilon^4}
\frac{\partial u_0}{\partial \tau}
+
\varepsilon
\frac{\partial u_1}{\partial t}
+
\frac{1}{\varepsilon^3}
\frac{\partial u_1}{\partial \tau}
+
\varepsilon^2
\frac{\partial u_2}{\partial t}
+
\frac{1}{\varepsilon^2}
\frac{\partial u_2}{\partial \tau}
+
\frac{1}{\varepsilon}
\frac{\partial u_3}{\partial \tau}
+
\frac{\partial u_4}{\partial \tau}
+
\varepsilon
\frac{\partial u_5}{\partial \tau}
+
\varepsilon^2
\frac{\partial u_6}{\partial \tau}
+o(\varepsilon^2)
\end{equation}
and
\eqref{nppr370}.
Thus, at order $1/\varepsilon^4$, we find the equation
\begin{equation}
\label{nppr880}
\frac{\partial u_0}{\partial\tau}-
\Delta_y (b_1 b_2 u_0)
=0
\,.
\end{equation}
Recalling that $b_2$ does not depend on
$y$,
from \eqref{nppr880}
we have that
$b_1(x,t,y) u_0(x,t,y,\tau)$ does not depend on $y$ and $\tau$,
thus we set
$v_0(x,t)=b_1(x,t,y) u_0(x,t,y,\tau)$.
Indeed, we write \eqref{nppr880} as an equation for $v_0$ (which, clearly, has uniqueness)
and note that
$v_0$, constant with respect to $\tau$ and $y$, solves such an equation.

Now, we pass directly to the $\varepsilon^0$ order equation.
From
\eqref{prob030}, \eqref{nppr860}, and \eqref{nppr370}, we have
\begin{equation}
\label{nppr830}
\frac{\partial u_0}{\partial t}
+
\frac{\partial u_4}{\partial \tau}
-
\Big[
     \Delta_x(b_1 b_2 u_0)
     +
     (\nabla_y\cdot\nabla_x
      +
      \nabla_x\cdot\nabla_y)
     (b u_1 + b_1 b_2 u_2)
     +
     \Delta_y(b u_3)
     +
     \Delta_y(b_1 b_2 u_4)
\Big]
=
0\,.
\end{equation}
Integrating on
$\mathcal{Q}$,
we find again \eqref{nppr550} and, with the same arguments as those
used above, we derive \eqref{nppr240}.

The last situation we discuss in this subsection
is the problem \eqref{prob030}--\eqref{prob050}
for $\alpha=2$ and $\beta=1$.
We, then, let $y=x/\varepsilon^2$
and
$\tau=t/\varepsilon$ and,
by abusing the notation, we write the differential rules
\begin{equation}
\label{nppr940}
\frac{\partial}{\partial t}=\frac{\partial}{\partial t}
+\frac{1}{\varepsilon}\frac{\partial}{\partial\tau},
\;\;
\nabla_x=\nabla_x+\frac{1}{\varepsilon^2}\nabla_y,
\textrm{ and }\;\;
\Delta_x=\Delta_x
+
\frac{1}{\varepsilon^2}
(
 \nabla_y\cdot\nabla_x
 +\nabla_x\cdot\nabla_y
)
+
\frac{1}{\varepsilon^4}\Delta_y\,,
\end{equation}
and we look for a solution of \eqref{prob030}, using the formal expansion
\eqref{nppr050}.
By substituting \eqref{nppr050} in \eqref{prob030}, we get
\begin{equation}
\label{nppr960}
\frac{\partial u_\varepsilon}{\partial t}
=
\frac{\partial u_0}{\partial t}
+
\frac{1}{\varepsilon}
\frac{\partial u_0}{\partial \tau}
+
\varepsilon
\frac{\partial u_1}{\partial t}
+
\frac{\partial u_1}{\partial \tau}
+
\varepsilon^2
\frac{\partial u_2}{\partial t}
+
\varepsilon
\frac{\partial u_2}{\partial \tau}
+
\varepsilon^2
\frac{\partial u_3}{\partial \tau}
+o(\varepsilon^2)
\end{equation}
and
\eqref{nppr370}.

Thus, at order $1/\varepsilon^4$, we find again the equation
\eqref{nppr380}, which leads, as above, to
$v_0(x,t,\tau)=b_1(x,t,y) u_0(x,t,y,\tau)$.

We now consider the $1/\varepsilon$ order equation.
From
\eqref{prob030}, \eqref{nppr960}, and \eqref{nppr370}, we have
\begin{equation}
\label{nppr990}
\frac{\partial u_0}{\partial \tau}
-
[
     (\nabla_x\cdot\nabla_y
     +
     \nabla_y\cdot\nabla_x)
     (b u_0+b_1 b_2 u_1)
     +
     \Delta_y(b u_2)
     +
     \Delta_y(b_1 b_2 u_3)
]
=0
.
\end{equation}
By integrating on $\mathcal{Y}$ and using that $v_0$ does not depend on $y$
and $b_1$ does not depend on $\tau$,
we arrive again to the compatibility condition \eqref{nppr450},
which implies that $v_0=v_0(x,t)$.

We finally consider the $\varepsilon^0$ order equation.
From
\eqref{prob030}, \eqref{nppr960}, and \eqref{nppr370}, we have
\begin{equation}
\label{nppr1030}
\frac{\partial u_0}{\partial t}
+
\frac{\partial u_1}{\partial \tau}
-
\Big[
     \Delta_x(b_1 b_2 u_0)
     +
     (\nabla_y\cdot\nabla_x
      +
      \nabla_x\cdot\nabla_y)
     (b u_1 + b_1 b_2 u_2)
     +
     \Delta_y(b u_3)
     +
     \Delta_y(b_1 b_2 u_4)
\Big]
=
0
.
\end{equation}
Integrating on
$\mathcal{Q}$,
we get once again \eqref{nppr550} and, with the same arguments as those
used above, we derive \eqref{nppr240}.

\section{Preliminary results}
\label{s:pre}
\par\noindent
In this Section, we always assume that $\alpha\leq 1$ and $v_{\varepsilon}$ is the solution to \eqref{fas023}, under the assumptions listed in Subsection~\ref{s:p-mod}.

\subsection{Estimates}
\label{s:estim}
\par\noindent
We collect here some estimates that will be used in the sequel.

\begin{lemma}
\label{t:enest}
There exists $\const>0$, depending on $T, \|f\|_2, \|\bar v_\varepsilon\|_2$ and the structural constants of the problem,
but independent of $\varepsilon$, such that
\begin{equation}
\label{fas100}
\sup_{t\in[0,T]}
\int_\Omega v_\varepsilon^2\,\emph{d}x
+
\int_0^T\int_\Omega |\nabla v_\varepsilon|^2\,\emph{d}x\,\emph{d}t
\le
\gamma\,.
\end{equation}
\end{lemma}

\par\noindent
\textit{Proof.}
Multiplying \eqref{fas000} by $v_\varepsilon/a^\varepsilon_2$
and integrating by parts, we obtain
\begin{equation}
\label{fas120}
\begin{split}
&
\frac{1}{2}
\int_\Omega v_\varepsilon^2\frac{a^\varepsilon_1}{b^\varepsilon_1} \mathrm{d}x
+
\int_0^t\int_\Omega
\frac{1}{a^2_\varepsilon}
\Big(b^\varepsilon_2+\varepsilon\frac{b^\varepsilon}{b^\varepsilon_1}\Big)
B^\varepsilon
\nabla v_\varepsilon\cdot\nabla v_\varepsilon
\,\mathrm{d}x\,\mathrm{d}s
\\
&
\quad
=
\int_0^t\int_\Omega
\Big[
     f\frac{v_\varepsilon}{a^\varepsilon_2}
     +\frac{1}{2}v_\varepsilon^2
      \Big\{
         \frac{1}{b^\varepsilon_1}\frac{\partial a^\varepsilon_1}{\partial s}
         -
         a^\varepsilon_1\frac{\partial}{\partial s}\frac{1}{b^\varepsilon_1}
      \Big\}
\Big]
\,\mathrm{d}x\,\mathrm{d}s
\\
&
\qquad
-
\int_0^t\int_\Omega
v_\varepsilon
\Big[
 \Big(b^\varepsilon_2+\varepsilon\frac{b^\varepsilon}{b^\varepsilon_1}\Big)
 B^\varepsilon
 \nabla v_\varepsilon\cdot\nabla\frac{1}{a^\varepsilon_2}
+
 \frac{1}{a^\varepsilon_2}
 B^\varepsilon
 \nabla v_\varepsilon\cdot
 \nabla
 \Big(b^\varepsilon_2+\varepsilon\frac{b^\varepsilon}{b^\varepsilon_1}\Big)
\Big]
\,\mathrm{d}x\,\mathrm{d}s
\\
&
\qquad
-
\int_0^t\int_\Omega
 v_\varepsilon^2
 B^\varepsilon
 \nabla\frac{1}{a^\varepsilon_2}\cdot
 \nabla
 \Big(b^\varepsilon_2+\varepsilon\frac{b^\varepsilon}{b^\varepsilon_1}\Big)
\,\mathrm{d}x\,\mathrm{d}s
+
\frac{1}{2}
\int_\Omega
\Big[
     v^2_\varepsilon
     \frac{a^\varepsilon_1}{b^\varepsilon_1}
\Big]_{s=0}
\,\mathrm{d}x
\,.
\end{split}
\end{equation}
Under our assumptions on the sign of the coefficients,
the left hand side of \eqref{fas120} can be bounded from
below by the left hand side of \eqref{fas100}.
Again appealing to our assumptions and, in particular,
to $\alpha\le1$, we see that all the functions
appearing in the integrals on the right hand side
of \eqref{fas120} are bounded by an absolute constant,
with the exception of $f$, $v^\varepsilon$, and $\nabla v^\varepsilon$.

Then, by Young inequality, the right hand side of \eqref{fas120} can be bounded
from above by
\begin{equation}
\label{fas130}
\gamma
\left(\|f\|_2^2
+
\delta
\int_0^t\int_\Omega
|\nabla v_\varepsilon|^2
\,\mathrm{d}x\,\mathrm{d}s
+
\frac{1}{\delta}
\int_0^t\int_\Omega
v_\varepsilon^2
\,\mathrm{d}x\,\mathrm{d}s
+
\|\bar v_\varepsilon\|_2^2\right)
\,,
\end{equation}
where $\const$ is independent of $\varepsilon$ and  $\delta>0$ can be chosen so that the gradient term can be
absorbed into the left hand side.
Finally, the result follows from the application of Gronwall lemma.
\qed

Taking into account that the initial datum $\bar u_\varepsilon$ (and, therefore, $\bar v_\varepsilon$)
belongs not only to the space $L^2(\Omega)$ (as needed in the previous estimate), but it is, indeed, in $H^1_0(\Omega_T)$,
we can obtain also some estimates for the time-derivative of the solution $v_\varepsilon$, as stated in the
next two lemmas.

\begin{lemma}
\label{t:vt}
There exists $\const>0$, depending on $T, \|f\|_2, \|\bar v_\varepsilon\|_2$ and the structural constants of the problem,
but independent of $\varepsilon$, such that
\begin{equation}
\label{fas140}
\int_0^T \int_\Omega
\Big(\frac{\partial v_\varepsilon}{\partial t}\Big)^2
\,\emph{d}x\,\emph{d}t
+
\sup_{t\in[0,T]}
\int_\Omega |\nabla v_\varepsilon|^2
\,\emph{d}x
\le
\frac{\gamma}{\varepsilon^\beta}
+\gamma
\int_\Omega
|\nabla \bar v_\varepsilon|^2
\,\emph{d}x\,.
\end{equation}
\end{lemma}

\par\noindent
\textit{Proof.}
Let us multiply \eqref{fas000} times $\partial v_\varepsilon/\partial t$
and integrate by parts to obtain
\begin{equation}
\label{fas150}
\begin{split}
&
\int_0^t\int_\Omega
\frac{a^\varepsilon_1 a^\varepsilon_2}{b^\varepsilon_1}
\Big(\frac{\partial v_\varepsilon}{\partial s}\Big)^2
\,\mathrm{d}x\,\mathrm{d}s
+
\frac{1}{2}
\int_\Omega
 \Big(b^\varepsilon_2+\varepsilon\frac{b^\varepsilon}{b^\varepsilon_1}\Big)
 B^\varepsilon
 \nabla v_\varepsilon
 \cdot
 \nabla v_\varepsilon
\,\mathrm{d}x
\\
&
\quad
=
 -
\int_0^t\int_\Omega
 a^\varepsilon_1 a^\varepsilon_2
 v_\varepsilon
 \Big(\frac{\partial v_\varepsilon}{\partial s}\Big)
 \frac{\partial}{\partial s}\frac{1}{b^\varepsilon_1}
 \,\mathrm{d}x\,\mathrm{d}s
+
\frac{1}{2}
\int_0^t\int_\Omega
\frac{\partial}{\partial s}
\Big[
    \Big(b^\varepsilon_2+\varepsilon\frac{b^\varepsilon}{b^\varepsilon_1}\Big)
    B^\varepsilon
\Big]
 \nabla v_\varepsilon
 \cdot
 \nabla v_\varepsilon
 \,\mathrm{d}x\,\mathrm{d}s
\\
&
\qquad
-
\int_0^t\int_\Omega
 v_\varepsilon
 B^\varepsilon
 \nabla
 \Big(b^\varepsilon_2+\varepsilon\frac{b^\varepsilon}{b^\varepsilon_1}\Big)
 \cdot
 \nabla
 \frac{\partial v_\varepsilon}{\partial s}
 \,\mathrm{d}x\,\mathrm{d}s
+
\int_0^t\int_\Omega
 f
 \frac{\partial v_\varepsilon}{\partial s}
 \,\mathrm{d}x\,\mathrm{d}s
\\
&
\qquad
+
\frac{1}{2}
\int_\Omega
\Big[
   \Big(b^\varepsilon_2+\varepsilon\frac{b^\varepsilon}{b^\varepsilon_1}\Big)
   B^\varepsilon
   \nabla v_\varepsilon
   \cdot
   \nabla v_\varepsilon
\Big]_{s=0}
 \,\mathrm{d}x
\\
&
\quad
=
I_1+I_2+I_3+I_4+I_5
\;.
\end{split}
\end{equation}
Under our assumptions on the sign of the coefficients,
the left hand side of \eqref{fas150} can be bounded from
below by the left hand side of \eqref{fas140}.
Next, we give estimates for each term $I_i$.
By Young inequality, we get
\begin{gather}
\label{fas160}
|I_1|
\le
\delta
\int_0^T\int_\Omega
\Big(\frac{\partial v_\varepsilon}{\partial t}\Big)^2
 \,\mathrm{d}x \,\mathrm{d}t
+
\frac{\gamma}{\delta}
\int_0^T\int_\Omega
 v_\varepsilon^2
 \,\mathrm{d}x \,\mathrm{d}t
\\
|I_2|
\le
\frac{\gamma}{\varepsilon^\beta}
\int_0^T\int_\Omega
 |\nabla v_\varepsilon|^2
 \,\mathrm{d}x\,\mathrm{d}t
\\
|I_4|
\le
\delta
\int_0^T\int_\Omega
\Big(\frac{\partial v_\varepsilon}{\partial t}\Big)^2
 \,\mathrm{d}x \,\mathrm{d}t
+
\frac{\gamma}{\delta}
\int_0^T\int_\Omega
 f^2
 \,\mathrm{d}x \,\mathrm{d}t
\\
|I_5|
\le
\gamma
\int_\Omega
|\nabla\bar v_\varepsilon|^2
 \,\mathrm{d}x
.
\end{gather}
Moreover, we calculate
\begin{equation}
\label{fas170}
I_3
=
-\int_\Omega
\Big[
 v_\varepsilon
 B^\varepsilon
 \nabla
 \Big(b^\varepsilon_2+\varepsilon\frac{b^\varepsilon}{b^\varepsilon_1}\Big)
 \cdot
 \nabla
 v_\varepsilon
\Big]_0^t
 \,\mathrm{d}x
+
\int_0^t\int_\Omega
\frac{\partial}{\partial s}
\Big[
 v_\varepsilon
 B^\varepsilon
 \nabla
 \Big(b^\varepsilon_2+\varepsilon\frac{b^\varepsilon}{b^\varepsilon_1}\Big)
\Big]
 \cdot
 \nabla
 v_\varepsilon
 \,\mathrm{d}x \,\mathrm{d}s
\end{equation}
and thus, recalling that $\alpha\le1$, we obtain
\begin{equation}
\label{fas180}
\begin{split}
|I_3|
&
\le
\gamma
\int_\Omega
 |\nabla v_\varepsilon|
 |v_\varepsilon|
 \,\mathrm{d}x
+
\gamma
\int_\Omega
 |\nabla\bar v_\varepsilon|
 |\bar v_\varepsilon|
 \,\mathrm{d}x
\\
&
\quad
+
\frac{\gamma}{\varepsilon^\beta}
\int_0^T \int_\Omega
 |\nabla v_\varepsilon|
 |v_\varepsilon|
 \,\mathrm{d}x \,\mathrm{d}t
+
\gamma
\int_0^T \int_\Omega
 \Big|\nabla v_\varepsilon\Big|
 \Big|\frac{\partial v_\varepsilon}{\partial t}\Big|
 \,\mathrm{d}x \,\mathrm{d}t
\;.
\end{split}
\end{equation}
Again, an application of Young inequality
gives
\begin{equation}
\label{fas190}
\begin{split}
|I_3|
&
\le
\delta
 \int_\Omega
 |\nabla v_\varepsilon|^2
 \,\mathrm{d}x
+
\frac{\gamma}{\delta}
 \int_\Omega
 v_\varepsilon^2
 \,\mathrm{d}x
+\gamma
 \int_\Omega
 (\bar v_\varepsilon^2
 +
 |\nabla \bar v_\varepsilon|^2)
 \,\mathrm{d}x
\\
&
\quad
+
\frac{\gamma}{\varepsilon^\beta}
\int_0^T\int_\Omega
 (v_\varepsilon^2
 +
 |\nabla v_\varepsilon|^2)
 \,\mathrm{d}x \,\mathrm{d}t
+
\delta
\int_0^T\int_\Omega
 \Big(\frac{\partial v_\varepsilon}{\partial t}\Big)^2
 \,\mathrm{d}x \,\mathrm{d}t
+
\frac{\gamma}{\delta}
\int_0^T\int_\Omega
 |\nabla v_\varepsilon|^2
 \,\mathrm{d}x \,\mathrm{d}t
\,.
\end{split}
\end{equation}
For $\delta$ suitably small, we can absorb the terms in
\eqref{fas190} multiplied by $\delta$ into
the left hand side of \eqref{fas150}.
Then, the claim follows by applying Lemma~\ref{t:enest}.
\qed

\begin{lemma}
\label{t:vtl}
Let $T_1\in (0, T)$. Then, there exists $\const>0$, depending on
$T_1,T,\|f\|_2, \|\bar v_\varepsilon\|_2$ and the structural constants of the problem,
but independent of $\varepsilon$, such that
\begin{equation}
\label{fas340}
\int_{T_1}^T \int_\Omega
\Big(\frac{\partial v_\varepsilon}{\partial t}\Big)^2
\,\emph{d}x\,\emph{d}t
+
\sup_{t\in[T_1,T]}
\int_\Omega |\nabla v_\varepsilon|^2
\,\emph{d}x
\le
\frac{\gamma}{\varepsilon^\beta}\,.
\end{equation}
\end{lemma}

\par\noindent
\textit{Proof.}
Let us multiply \eqref{fas000} times
$\phi(t)\partial v_\varepsilon/\partial t$,
where $\phi(t)\in C^\infty(\mathbb{R})$ such that
$\phi(t)=0$ for $t\le T_1/2$,
$\phi(t)=1$ for $t> T_1$,
and
$0\le\phi'(t)\le 4/T_1$,
and integrate by parts to obtain
\begin{displaymath}
\label{fas350}
\begin{split}
&
\int_0^t\int_\Omega
\phi
\frac{a^\varepsilon_1 a^\varepsilon_2}{b^\varepsilon_1}
\Big(\frac{\partial v_\varepsilon}{\partial s}\Big)^2
\,\mathrm{d}x\,\mathrm{d}s
+
\frac{1}{2}
\int_\Omega
\phi
 \Big(b^\varepsilon_2+\varepsilon\frac{b^\varepsilon}{b^\varepsilon_1}\Big)
 B^\varepsilon
 \nabla v_\varepsilon
 \cdot
 \nabla v_\varepsilon
\,\mathrm{d}x
\\
&
\quad
=
-
\int_0^t\int_\Omega
 \phi
 a^\varepsilon_1 a^\varepsilon_2
 v_\varepsilon
 \Big(\frac{\partial v_\varepsilon}{\partial s}\Big)
 \frac{\partial}{\partial s}\frac{1}{b^\varepsilon_1}
 \,\mathrm{d}x\,\mathrm{d}s
+
\frac{1}{2}
\int_0^t\int_\Omega
\frac{\partial}{\partial s}
\Big[
     \phi
     \Big(b^\varepsilon_2+\varepsilon\frac{b^\varepsilon}{b^\varepsilon_1}\Big)
     B^\varepsilon
\Big]
 \nabla v_\varepsilon
 \cdot
 \nabla v_\varepsilon
 \,\mathrm{d}x\,\mathrm{d}s
\\
&
\qquad
-
\int_0^t\int_\Omega
 \phi
 v_\varepsilon
 B^\varepsilon
 \nabla
 \Big(b^\varepsilon_2+\varepsilon\frac{b^\varepsilon}{b^\varepsilon_1}\Big)
 \cdot
 \nabla
 \frac{\partial v_\varepsilon}{\partial s}
 \,\mathrm{d}x\,\mathrm{d}s
+
\int_0^t\int_\Omega
 \phi
 f
 \frac{\partial v_\varepsilon}{\partial s}
 \,\mathrm{d}x\,\mathrm{d}s
\\
&
\quad
=
I_1+I_2+I_3+I_4
\;.
\end{split}
\end{displaymath}
Now, the terms $I_1$ and $I_4$ are treated as in the proof
of Lemma~\ref{t:vt}.
For $I_2$, we write
\begin{gather}
\label{fas360}
|I_2|
\le
\frac{\gamma(1+T_1^{-1})}{\varepsilon^\beta}
\int_0^T\int_\Omega
 |\nabla v_\varepsilon|^2
 \,\mathrm{d}x\,\mathrm{d}t
\;.
\end{gather}
Moreover, we calculate
\begin{equation}
\label{fas370}
I_3
=
-\int_\Omega
\Big[
 \phi
 v_\varepsilon
 B^\varepsilon
 \nabla
 \Big(b^\varepsilon_2+\varepsilon\frac{b^\varepsilon}{b^\varepsilon_1}\Big)
 \cdot
 \nabla
 v_\varepsilon
\Big]_0^t
 \,\mathrm{d}x
+
\int_0^t\int_\Omega
\frac{\partial}{\partial s}
\Big[
 \phi
 v_\varepsilon
 B^\varepsilon
 \nabla
 \Big(b^\varepsilon_2+\varepsilon\frac{b^\varepsilon}{b^\varepsilon_1}\Big)
\Big]
 \cdot
 \nabla
 v_\varepsilon
 \,\mathrm{d}x \,\mathrm{d}s
\end{equation}
and thus, recalling that $\alpha\le1$, we obtain
\begin{multline}
\label{fas380}
|I_3|
\le
\gamma
\int_\Omega
 |\nabla v_\varepsilon|
 |v_\varepsilon|
 \,\mathrm{d}x
+
\frac{\gamma(1+T_1^{-1})}{\varepsilon^\beta}
\int_0^T \int_\Omega
 |\nabla v_\varepsilon|
 |v_\varepsilon|
 \,\mathrm{d}x \,\mathrm{d}t
\\
+
\gamma
\int_0^T \int_\Omega
 \phi
 \Big|\nabla v_\varepsilon\Big|
 \Big|\frac{\partial v_\varepsilon}{\partial t}\Big|
 \,\mathrm{d}x \,\mathrm{d}t
\;.
\end{multline}
As in the proof of Lemma~\ref{t:vt},
a final application of the Young inequality
yields \eqref{fas340}.
\qed

\begin{proposition}
 \label{p:time_comp}
For any $0<\delta<\maxT/2$, there exists $\const>0$ (depending on $T,\Vert f\Vert_2, \Vert \bar v_\varepsilon\Vert_2$,
the structural constants of the problem and $\delta$), such that
 \begin{equation}
   \label{eq:time_comp}
   \int_{\delta}^{\maxT-\delta}
   \int_{\Oset}
   \Abs{\ve(x,t+h)-\ve(x,t)}^2
   \di x
   \di t
   \le
   \const
   \sqrt{h}
   \,,
 \end{equation}
for any $0<h<\delta/2$.
\end{proposition}
 \begin{proof}
   We select, as a test function in the integral formulation \eqref{fas023},
   the function
   \begin{equation*}
     \phi(x,t)=\frac{
       \phj(x,t)
     }{
       \adue(x,t)
     }
     \,,
          \qquad
     \hbox{with}\ \
     \phj\in H^1_0(\Omega_T)
     \,.
   \end{equation*}
   We obtain
   \begin{multline}
     \label{eq:est_weak}
     \iint_{\Ocyl}
     \Big\{
     -
     \frac{\ve}{\bune}
     \pder{}{t}
     (\aune\phj)
     +
     \Bme
     \grad\fpcov
     \scpr
     \grad\Big(
     \frac{\phj}{\adue}
     \Big)
     \Big\}
     \di x
     \di t
     =
     \iint_{\Ocyl}
     f
     \frac{\phj}{\adue}
     \di x
     \di t
     \,.
   \end{multline}
   Here, for any $F=F(x,t)$, we denote by $\tshift{F}(x,t)=F(x,t+h)$ its
   time shift. Let $\delta\in(0,\maxT/2)$, $0<h<\delta/2$, and
assume that $\phj(x,t)=0$ for $t<\delta/2$ and for
   $t>T-\delta/2$. Using the formula \eqref{eq:est_weak} with $\phj(x,t)$ replaced with $\phj(x,t-h)$, and then changing variables to $(x,t+h)$, but still keeping the old variable names, we obtain
      \begin{multline}
     \label{eq:est_weak_shift}
     \iint_{\Ocyl}
     \Big\{
     -
     \frac{\tshift{\ve}}{\tshift{\bune}}
     \pder{}{t}
     (\tshift{\aune}\phj)
     +
     \tshift{\Bme}
     \grad\tsfpcov
     \scpr
     \grad\Big(
     \frac{\phj}{\tshift{\adue}}
     \Big)
     \Big\}
     \di x
     \di t
     =
     \iint_{\Ocyl}
     \tshift{f}
     \frac{\phj}{\tshift{\adue}}
     \di x
     \di t
     \,.
   \end{multline}
   Next, in \eqref{eq:est_weak}--\eqref{eq:est_weak_shift}, we select $\phj=\phj_{h}$ where
   \begin{equation*}
     \varphi_h(x,t)
     =
     -
     \zeta(t)
     \int_t^{t+h}
     \ve(x,s)
     \di s
     \,,
   \end{equation*}
   where $\zeta\in {\mathcal C}^1_0{(\delta/2,\maxT-\delta/2)}$ is a nonnegative function such that
   $\zeta=1$ in $(\delta,\maxT-\delta)$ and $\abs{\zeta'}\le
   \const/\delta$.
   \\
   On subtracting the two integral formulations \eqref{eq:est_weak} and \eqref{eq:est_weak_shift}, we obtain
   \begin{equation}
     \label{eq:time_comp_ii}
     \begin{split}
     &\iint_{\Ocyl}
     \Big\{
     -
     \Big[
     \frac{\tshift{\ve}}{\tshift{\bune}}
     \tshift{\aune}
     -
     \frac{\ve}{\bune}
     \aune
     \Big]
     \Big\}
     \pder{\varphi_h}{t}
     \di x
     \di t
     +
     \iint_{\Ocyl}
     \Big\{
     -
     \Big[
     \frac{\tshift{\ve}}{\tshift{\bune}}
     \pder{\tshift{\aune}}{t}
     -
     \frac{\ve}{\bune}
     \pder{\aune}{t}
     \Big]
     \Big\}
     \varphi_h
     \di x
     \di t
     \\
     &\quad
     +
     \iint_{\Ocyl}
     \Big\{
     \frac{1}{\tshift{\adue}}
     \tshift{\Bme}
     \grad\tsfpcov
     -
     \frac{1}{\adue}
     \Bme
     \grad\fpcov
     \Big\}
     \scpr
     \grad\phj_{h}
     \di x
     \di t
     \\
     &\quad+
     \iint_{\Ocyl}
     \Big\{
     \tshift{\Bme}
     \grad\tsfpcov
     \scpr
     \grad\frac{1}{\tshift{\adue}}
     -
     \Bme
     \grad\fpcov
     \scpr
     \grad
     \frac{1}{\adue}
     \Big\}
     \phj_{h}
     \di x
     \di t
     \\
     &\qquad=
     \iint_{\Ocyl}
     \Big\{
     \frac{\tshift{f}}{\tshift{\adue}}
     -
     \frac{f}{\adue}
     \Big\}
     \phj_{h}
     \di x
     \di t
     \,.
     \end{split}
   \end{equation}
   For the sake of notational simplicity, we denote each integral with a different symbol, thereby rewriting \eqref{eq:time_comp_ii} as
   \begin{equation*}
     I_{1}
     +
     I_{2}
     +
     I_{3}
     +
     I_{4}
     =
     I_{5}
     \,,
   \end{equation*}
   where, actually, only the estimation of $I_{1}$ requires a detailed calculation.
   Indeed,
   \begin{multline}
     \label{eq:time_comp_iii}
     I_{1}
     =
     \iint_{\Ocyl}
     \Big[
     \frac{\tshift{\ve}}{\tshift{\bune}}
     \tshift{\aune}
     -
     \frac{\ve}{\bune}
     \aune
     \Big]
     \left\{
       \zeta[\tshift{\ve}-\ve]
+
       \zeta'
       \int_t^{t+h}\ve(x,s)\di s
     \right\}
     \di x
     \di t
     \\
     =
     \iint_{\Ocyl}
     [\tshift{\ve}-\ve]^2
     \frac{\zeta\aune}{\bune}
     \di x
     \di t
     +
     \iint_{\Ocyl}
     \tshift{\ve}[\tshift{\ve}-\ve]
     \left[
     \frac{\tshift{\aune}}{\tshift{\bune}}
     -
     \frac{\aune}{\bune}
     \right]
     \zeta
     \di x
     \di t
     \\
     +
     \iint_{\Ocyl}
     \left[\left(
       \frac{\tshift{\ve}}{\tshift{\bune}}
       \tshift{\aune}
       -
       \frac{\ve}{\bune}
       \aune
     \right)
     \zeta'
     \int_t^{t+h}\ve(x,s)\di s\right]
     \di x
     \di t
     =
     I_{11}
     +
     I_{12}
     +
     I_{13}
     \,.
   \end{multline}
   The term $I_{11}$
   essentially equals the one estimated in the statement.  The
   term $I_{12}$ is
   estimated, invoking the time regularity of $\aun$, $\bun$, by
   \begin{multline}
     \label{eq:time_comp_iv_b}
     \abs{I_{12}}
     \le
     \const
     \int_{\Oset}
     \int_{\delta/2}^{\maxT-\delta/2}
     \Abs{\tshift{\ve}}
     \big(
     \Abs{\tshift{\ve}}+\Abs{\ve}
     \big)
     \big[
     \Abs{
       \tshift{\aune}
       -
       \aune
     }
     +
     \Abs{
       \tshift{\bune}
       -
       \bune
     }
     \big]
     \di x
     \di t
     \le
     \const
     \Norma{\ve}{2}^2
     h
     \,.
   \end{multline}
   The integral $I_{13}$ can be bounded by means of
   the
   H\"older inequality as follows
   \begin{multline}
     \label{eq:time_comp_iv}
     \abs{I_{13}}
     \le
     \const
     \Norma{\zeta^{\prime}}{\oo}
     \left(
       \int_{\Oset}\int_{\delta/2}^{\maxT-\delta/2}
       \Abs{\tshift{\ve}}^2+\Abs{\ve}^2
       \di x
       \di t
     \right)^{\frac{1}{2}}
     \left(
       \int_{\Oset}\int_{\delta/2}^{\maxT-\delta/2}
       \Abs{\int_t^{t+h}\ve(x,s)\di s}^2
       \di x
       \di t
     \right)^{\frac{1}{2}}
     \\
     \le
     \frac{\const}{\delta}
     \Norma{\ve}{2}^2
     \sqrt{h}
     \,.
   \end{multline}
   Clearly, the integrals $I_{2}$, $I_{3}$, $I_{4}$ and $I_{5}$ can be
   estimated by means of a similar device, once we remark that, owing to
   the assumed regularity in space of $\bdu$, $\bun$, $\bp$, we get
   \begin{equation}
     \label{eq:time_comp_k}
     \Abs{
       \grad\fpco
     }
     \le
     \abs{\grad\bdue}
     +
     \varepsilon
     \Abs{
       \grad
       \frac{\bpe}{\bune}
     }
     \le
     \gamma
     \,.
   \end{equation}
   For example,  the integral $I_{3}$ can be estimated by
   \begin{multline}
     \label{eq:time_comp_v}
     \abs{I_{3}}
     \le
     \const
     \int_{\Oset}\int_{\delta/2}^{\maxT-\delta/2}
     \big(
     \Abs{\ve}
     +
     \Abs{\grad\ve}
     +
     \Abs{\tshift{\ve}}
     +
     \Abs{\grad\tshift{\ve}}
     \big)
     \Abs{\int_t^{t+h}\grad\ve(x,s)\di s}
     \di x
     \di t
     \\
     \le
     \const
     \big(
     \norma{\ve}{2}^{2}
     +
     \norma{\grad\ve}{2}^{2}
     \big)
     \sqrt{h}
     \,.
   \end{multline}
   Finally, on collecting all the estimates above, we get
   \eqref{eq:time_comp}.
 \end{proof}

\subsection{Unfolding}
\label{s:unfold}
\par\noindent
In the sequel, we denote by $[r]$ the integer part of $r\in\mathbb{R}$ and,
for $x\in\mathbb{R}^n$, we define the vector with integer components
$[x]=([x_1],\dots,[x_n])$.

Let us consider the tiling of $\mathbb{R}^n$
given by the boxes $\varepsilon^\alpha(\xi+\mathcal{Y})$,
with $\xi\in\mathbb{Z}^n$.
Following \cite{AAB2017},
we set
\begin{equation}
\label{fold000}
\Xi_\varepsilon
=
\{\xi\in\mathbb{Z}^n:\,
  \varepsilon^\alpha(\xi+\mathcal{Y})\subset\Omega\},
\quad
\hat\Omega_\varepsilon
=
\text{interior}
\Big\{
      \bigcup_{\xi\in\Xi_\varepsilon}
           \varepsilon^\alpha(\xi+\overline{\mathcal{Y}})
\Big\}
,
\end{equation}
and
\begin{equation}
\label{fold010}
\hat T_\varepsilon
=
\Big\{
      t\in(0,T):\,
      \varepsilon^\beta
      \Big(
           \Big[\frac{t}{\varepsilon^\beta}\Big]+1
      \Big)
      \le T
\Big\}
,
\quad
\Lambda_\varepsilon
=
\hat\Omega_\varepsilon\times\hat T_\varepsilon
\;.
\end{equation}
We introduce also the space-time cell containing the point $(x,t)$ as
$$
\mathcal{Q}_\eps(x,t)=\eps^\alpha\left(\left[\frac{x}{\eps^\alpha}\right]+\mathcal{Y}\right)
\times\eps^\beta\left(\left[\frac{t}{\eps^\beta}\right]+\mathcal{S}\right)\,.
$$

\begin{definition}
\label{d:fold010}
The time--periodic unfolding operator $\mathcal{T}_\varepsilon$
of a Lebesgue measurable function $w$ defined on $\Omega_T$
is given by
\begin{equation}
\label{fold020b}
\mathcal{T}_\varepsilon(w)(x,t,y,\tau)
=
\left\{
\begin{alignedat}{2}
&
w\Big(\varepsilon^\alpha\Big[\frac{x}{\varepsilon^\alpha}\Big]
     +\varepsilon^\alpha y,
      \varepsilon^\beta\Big[\frac{t}{\varepsilon^\beta}\Big]
     +\varepsilon^\beta \tau
 \Big)
,
&
\quad
&
(x,t,y,\tau)\in\Lambda_\varepsilon\times\mathcal{Q},
\\
&
0,
&
\quad
&
\text{otherwise}
.
\end{alignedat}
\right.
\end{equation}
\end{definition}

Note that, by definition, it easily follows that
\begin{equation}
\label{fold080}
\mathcal{T}_\varepsilon(w_1w_2)
=
\mathcal{T}_\varepsilon(w_1)
\mathcal{T}_\varepsilon(w_2)\,.
\end{equation}

\begin{definition}
\label{d:fold000}
The space--time average operator $\mathcal{M}_\varepsilon$
of a Lebesgue integrable function $w$ defined on $\Omega_T$
is given by
\begin{equation}
\label{fold020}
\mathcal{M}_\varepsilon(w)(x,t)
=
\left\{
\begin{alignedat}{2}
&
\frac{1}{\eps^{N\alpha+\beta}}\int_{\mathcal{Q}_\eps(x,t)}w(\zeta,s)\di \zeta\di s
,
&
\quad
&
(x,t)\in\Lambda_\varepsilon,
\\
&
0,
&
\quad
&
\text{otherwise}
.
\end{alignedat}
\right.
\end{equation}
Moreover, the space--time oscillation operator
is defined as
\begin{equation}
\label{fold070}
\mathcal{Z}_\varepsilon(w)(x,t,y,\tau)
=
\mathcal{T}_\varepsilon(w)(x,t,y,\tau)
-
\mathcal{M}_\varepsilon(w)(x,t)
\;.
\end{equation}
\end{definition}

Notice that, by a simple change of variables, it easily follows that
\begin{equation}
\label{fold050}
\mathcal{M}_\varepsilon(w)(x,t)
=
\int_\mathcal{Q}
  \mathcal{T}_\varepsilon(w)(x,t,y,\tau)
\,\text{d}y\,\text{d}\tau\,
\,.
\end{equation}
Finally,
we denote by $\mathcal{M}_\mathcal{S}$
the microscopic time average
of an integrable function $\phi(x,t,y,\tau)$, i.e.
\begin{equation}
\label{fold060}
\mathcal{M}_\mathcal{S}(\phi)(x,t,y)
=
\int_\mathcal{S}
 \phi(x,t,y,\tau)
\,\text{d}\tau
\,.
\end{equation}
We conclude this section, recalling the following result (see \cite[Remark~2.9]{AAB2017}).

\begin{proposition}
\label{fold090}
For $\phi\in L^2(\mathcal{Q};{\mathcal C}(\overline\Omega_T))$
or
$\phi\in L^2(\Omega_T;{\mathcal C}(\overline{\mathcal{Q}}))$,
denote again by $\phi$ its
extension by $\mathcal{Q}$--periodicity
to $\Omega_T\times\mathbb{R}^{n+1}$
and set
$\phi_\varepsilon(x,t)=\phi(x,t,\varepsilon^{-\alpha}x,\varepsilon^{-\beta}t)$.
Then,
$\mathcal{T}_\varepsilon(\phi_\varepsilon) \to \phi$
strongly in $L^2(\Omega_T\times\mathcal{Q})$.
\end{proposition}

For later use, we define the functional spaces
\begin{equation}\label{d:space}
\begin{aligned}
& \Hper(\mathcal{Y})=\{v\in H^1_\textup{loc}(\mathbb{R}^n):\ \hbox{$v$ is $\mathcal{Y}$-periodic}\},
\cr
& \Hper(\mathcal{Q})=\{v\in H^1_\textup{loc}(\mathbb{R}^{n+1}):\ \hbox{$v$ is $\mathcal{Q}$-periodic}\}.
\end{aligned}
\end{equation}

\section{Homogenization}
\label{s:hom}
\par\noindent
In this section,
$u_{\varepsilon}$ and $v_{\varepsilon}$
are the solutions of problem \eqref{prob030}--\eqref{prob050}
and \eqref{fas000}--\eqref{fas020} in Subsection~\ref{s:p-mod}, and we assume all the hypotheses listed there.
As in Section \ref{s:pre}, we always assume $\alpha\le 1$.

We remark that, in all the cases we deal with, the final structure of
the macroscopic homogenized equation
will be the same, though the coefficients in it have to be
defined case--by--case.
Results are presented in two subsections:
Section~\ref{s:fast} is devoted to the case $\beta\ge2\alpha$
(fast oscillations),
while in Section~\ref{s:slow} the case $\beta<2\alpha$
(slow oscillations) is studied.

In each case we prove two theorems, the first
states the homogenization result and gives
the limit two--scale system, while the second one
introduces the corrector factorization and the resulting
single scale equation.
For technical reasons,
the uniqueness of the solutions
of the two limit problems is dealt
in the corollaries following the theorems.

\subsection{Fast oscillations}
\label{s:fast}
\par\noindent
Here, we treat the cases where $\beta\ge2\alpha$,
distinguishing between $\beta=2\alpha$ and $\beta>2\alpha$.

\begin{theorem}
\label{t:b=2aperv}
Let $\beta=2\alpha$.
Then, there exist
$v\in L^2(0,T;H^1_0(\Omega))$
and
$v_1\in L^2(\Omega_T;\Hper(\mathcal{Q}))$,
with $\int_\mathcal{Q} v_1 \,\emph{d}y\,\emph{d}\tau=0$,
such that
\begin{alignat}{2}
\label{fas030}
&
v_\varepsilon\rightharpoonup v\,,
\qquad
&
\emph{ weakly in } L^2(\Omega_T);
\\
\label{fas040}
&
v_\varepsilon\rightharpoonup v\,,
\qquad
&
\emph{ weakly in } L^2(0,T;H^1_0(\Omega));
\\
\label{fas050}
&
\mathcal{T}_\varepsilon(v_\varepsilon)
\rightharpoonup v\,,
\qquad
&
\emph{ weakly in } L^2(\Omega_T;H^1(\mathcal{Q}));
\\
\label{fas060}
&
\mathcal{T}_\varepsilon(\nabla v_\varepsilon)
\rightharpoonup
\nabla v+\nabla_y v_1\,,
\qquad
&
\emph{ weakly in } L^2(\Omega_T\times\mathcal{Q});
\\
\label{fas062}
&
\varepsilon^\alpha
\mathcal{T}_\varepsilon
\Big(
     \frac{\partial v_\varepsilon}{\partial t}
\Big)
\rightharpoonup
\frac{\partial v_1}{\partial\tau}\,,
\qquad
&
\emph{ weakly in } L^2(\Omega_T\times\mathcal{Q}).
\end{alignat}
Moreover,
the pair $(v,v_1)$ is a weak solution
of the two--scale problem
\begin{alignat}{2}
&
\int_{\mathcal{Q}}
\Big[
a_1
\frac{\partial}{\partial t}
\Big(
     \frac{v}{b_1}
\Big)
& \!\!\!\!\!\!\!\!\!\!\!\!\!\!\!\!\!\!\!\!\! &
\nonumber
\\
\label{fas070}
&
-
\frac{1}{a_2}
\emph{div}
\Big(
B
 \Big(
      b_2(\nabla v+\nabla_y v_1)
      +v\nabla b_2
      +\omega_{\alpha,1}v\nabla_y\Big(\frac{b}{b_1}\Big)
 \Big)
\Big)
\Big]
\,\emph{d}y\,\emph{d}\tau
=
f\!\!
\int_{\mathcal{S}}
\frac{\emph{d}\tau}{a_2}\,,
\!\!\!\!\!\!\!\!\!\!\!\!\!\!\!\!\!\! &
\emph{ in } \Omega_T\,;
&
\\
\label{fas075}
& 
\frac{a_1}{b_1}\frac{\partial v_1}{\partial \tau}
-
\frac{1}{a_2}
\emph{div}_y
\Big[
B
 \Big(
      b_2(\nabla v+\nabla_y v_1)
      +v\nabla b_2
      +\omega_{\alpha,1}v\nabla_y\Big(\frac{b}{b_1}\Big)
 \Big)
\Big]
=0\,,
\!\!\!\!\!\!\!\!\!\!\!\!\!\!
& \emph{ in } \Omega_T\times\mathcal{Q}\,;
&
\\
\label{fas080}
&
v=0\,,
\!\!\!\!\!\!\!\!\!\!\!\!\!\!\!\!\!\!\!\!\!\!\!\!\!\!\!\!\!\!\!\!\!\!\!\!\!\!\!\!\!\!\!\!\!\!\!\!\!\!\!\!\!\!\!\!
& \emph{ on } \partial\Omega\times(0,T)\,;
&
\\
\label{fas090}
&
v(x,0)=\bar u(x)\left(\int_{\mathcal{Y}}a_1(x,0,y)\,\emph{d}y\right)
       \Big(
            \int_{\mathcal{Y}}
              \frac{a_1(x,0,y)}{b_1(x,0,y)}
\,\emph{d}y
       \Big)^{-1}\,,
\!\!\!\!\!\!\!\!\!\!\!\!\!\!\!\!\!\!\!\!\!
& \emph{ in } \Omega
\;.
&
\end{alignat}
\end{theorem}

\par\noindent
\textit{Proof}
The convergence results in \eqref{fas030} and \eqref{fas040}
follow from the energy estimate \eqref{fas100};
\eqref{fas050} follows from \eqref{fas140}
and
\cite[Proposition~2.12]{AAB2017}
with $m=1/2$
and
by
replacing
$\tau$ with $\varepsilon^{2\alpha}$.
Finally,
\eqref{fas060}
and
\eqref{fas062}
follow from
\eqref{fas140}
and
\cite[Theorem~2.18]{AAB2017}
with $m=1/2$
by
replacing
$\tau$ with $\varepsilon^{2\alpha}$
and
$\varepsilon$ with $\varepsilon^\alpha$.

Now,
we choose, as test function in \eqref{fas023}, 
$\phi_\varepsilon(x,t)=\varphi(x,t)/a_2(x,t,t/\varepsilon^{2\alpha})$,
where $\varphi\in \mathcal{C}^\infty(\overline\Omega_T)$ with $\varphi(x,T)=$ in $\overline\Omega$ and $\varphi=0$ on
$\partial\Omega\times[0,T]$,
and we unfold the resulting equation.
We obtain
\begin{multline}
\label{pfas000}
-\int_0^T\int_\Omega
 \int_\mathcal{Q}
\mathcal{T}_\varepsilon(v_\varepsilon)
\mathcal{T}_\varepsilon\Big(\frac{1}{b^\varepsilon_1}\Big)
\mathcal{T}_\varepsilon\Big(
a^\varepsilon_1\frac{\partial\varphi}{\partial t}
+\varphi \frac{\partial a^\varepsilon_1}{\partial t}
\Big)
\,\textrm{d}y\,\textrm{d}\tau
\,\textrm{d}x\,\textrm{d}t
\\
+
\int_0^T\int_\Omega
\int_\mathcal{Q}
\mathcal{T}_\varepsilon(B^\varepsilon)
\mathcal{T}_\varepsilon\Big(
\nabla\Big(
 \Big(b_2^\varepsilon+\varepsilon\frac{b^\varepsilon}{b_1^\varepsilon}\Big)
 v_\varepsilon\Big)
\Big)
\cdot
\mathcal{T}_\varepsilon
   \Big(\nabla\Big(\frac{\varphi}{a^\varepsilon_2}\Big)\Big)
\,\textrm{d}y\,\textrm{d}\tau
\,\textrm{d}x\,\textrm{d}t
\\
=
\int_0^T\int_\Omega
f\frac{\varphi}{a^\varepsilon_2}
\,\textrm{d}x\,\textrm{d}t
+
\int_\Omega
\bar{u}
a^\varepsilon_1(x,0)
\varphi(x,0)
\,\textrm{d}x
+R_\varepsilon
\;,
\end{multline}
where $R_\varepsilon\to0$ for $\varepsilon\to0$.

We first note that
\begin{multline}
\label{pfas010}
\mathcal{T}_\varepsilon\Big(
\nabla\Big(
 \Big(b_2^\varepsilon+\varepsilon\frac{b^\varepsilon}{b_1^\varepsilon}\Big)
 v_\varepsilon\Big)
\Big)
=
\mathcal{T}_\varepsilon\Big(
\Big(b_2^\varepsilon+\varepsilon\frac{b^\varepsilon}{b_1^\varepsilon}\Big)
\nabla v_\varepsilon
\Big)
\\
+
\mathcal{T}_\varepsilon\Big(
v_\varepsilon
\Big(
\nabla b^\varepsilon_2
+\varepsilon\nabla_x\Big(\frac{b^\varepsilon}{b^\varepsilon_1}\Big)
+\varepsilon^{1-\alpha}
  \nabla_y \Big(\frac{b^\varepsilon}{b^\varepsilon_1}\Big)
\Big)
\Big)
\,.
\end{multline}
Therefore, passing to the limit $\varepsilon\to0$ in \eqref{pfas000},
and taking into account \eqref{fas050} and \eqref{fas060},
we get
\begin{multline}
\label{pfas030}
-\int_0^T\int_\Omega
 \int_\mathcal{Q}
\frac{v}{b_1}
\frac{\partial}{\partial t}(a_1\varphi)
\,\textrm{d}y\,\textrm{d}\tau
\,\textrm{d}x\,\textrm{d}t
\\
+
\int_0^T\int_\Omega
 \int_\mathcal{Q}
 B\Big(
   b_2(\nabla v+\nabla_y v_1)
   +v\Big(\nabla b_2+\omega_{\alpha,1}\nabla_y\Big(\frac{b}{b_1}\Big)\Big)
  \Big)
\cdot
\nabla\Big(\frac{\varphi}{a_2}\Big)
\,\textrm{d}y\,\textrm{d}\tau
\,\textrm{d}x\,\textrm{d}t
\\
=
\int_0^T\int_\Omega
f\varphi
\Big(
\int_\mathcal{S}
\frac{\mathrm{d}\tau}{a_2(x,t,\tau)}
\Big)
\,\textrm{d}x\,\textrm{d}t
+
\int_\Omega
\int_\mathcal{Y}
\bar{u}
a_1(x,0,y)
\varphi(x,0)
\,\textrm{d}y
\,\textrm{d}x
\,,
\end{multline}
i.e., the weak formulation of \eqref{fas070}
and \eqref{fas090}.

Next, we choose
$\phi_\varepsilon(x,t)=
  \varepsilon^\alpha(\varphi(x,t)/a_2(x,t,t/\varepsilon^{2\alpha}))
                     \psi(x/\varepsilon^\alpha,
                                     t/\varepsilon^{2\alpha})$,
where $\varphi\in \mathcal{C}^\infty(\overline\Omega_T)$ with $\varphi=0$ on
$\partial\Omega\times[0,T]$, and $\psi\in \Hper(\mathcal{Q})$,
as test function in \eqref{fas023} (where we do not integrate by parts in time)
and we unfold the resulting equation. We obtain
\begin{multline}
\label{pfas050}
 \varepsilon^\alpha
 \int_0^T\int_\Omega
 \int_\mathcal{Q}
\mathcal{T}_\varepsilon(a^\varepsilon_1)
\Big(
     \mathcal{T}_\varepsilon\Big(\frac{\partial v_\varepsilon}
                                      {\partial t}\Big)
     \mathcal{T}_\varepsilon\Big(\frac{1}
                                      {b^\varepsilon_1}\Big)
     +
     \mathcal{T}_\varepsilon(v_\varepsilon)
     \mathcal{T}_\varepsilon\Big(\frac{\partial}
                                      {\partial t}
                                 \frac{1}{b^\varepsilon_1}\Big)
\Big)
\mathcal{T}_\varepsilon(\varphi\psi)
\,\textrm{d}y\,\textrm{d}\tau
\,\textrm{d}x\,\textrm{d}t
\\
+
\varepsilon^\alpha
\int_0^T\int_\Omega
\int_\mathcal{Q}
\mathcal{T}_\varepsilon(B^\varepsilon)
\Big(
     \mathcal{T}_\varepsilon(v_\varepsilon)
     \mathcal{T}_\varepsilon\Big(
           \nabla_xb^\varepsilon_2
           +\varepsilon\nabla_x\Big(\frac{b^\varepsilon}
                                         {b^\varepsilon_1}\Big)
           +\varepsilon^{1-\alpha}
                      \nabla_y\Big(\frac{b^\varepsilon}
                                        {b^\varepsilon_1}\Big)
                            \Big)
     +
     \mathcal{T}_\varepsilon\Big(b^\varepsilon_2
                                 +\varepsilon\frac{b^\varepsilon}
                                                  {b^\varepsilon_1}\Big)
     \mathcal{T}_\varepsilon(\nabla v_\varepsilon)
\Big)
\\
\cdot
\Big(
     \mathcal{T}_\varepsilon\Big(\nabla\Big(\frac{\varphi}{a_2}\Big)\Big)
     \mathcal{T}_\varepsilon(\psi)
     +
     \mathcal{T}_\varepsilon\Big(\frac{\varphi}{a_2}\Big)
     \mathcal{T}_\varepsilon(\nabla_x\psi)
     +
     \frac{1}{\varepsilon^\alpha}
     \mathcal{T}_\varepsilon\Big(\frac{\varphi}{a_2}\Big)
     \mathcal{T}_\varepsilon(\nabla_y\psi)
\Big)
\,\textrm{d}y\,\textrm{d}\tau
\,\textrm{d}x\,\textrm{d}t
\\
=
\varepsilon^\alpha
\int_0^T\int_\Omega
f \frac{\varphi}{a_2} \psi
\,\textrm{d}x\,\textrm{d}t
+R_\varepsilon
\;,
\end{multline}
where $R_\varepsilon\to0$ for $\varepsilon\to0$.
Now, passing to the limit $\varepsilon\to0$
and taking into account \eqref{fas060} and \eqref{fas062},
we get
\begin{multline}
\label{pfas060}
\int_0^T\int_\Omega
\int_\mathcal{Q}
\frac{a_1}{b_1}
\frac{\partial v_1}{\partial\tau}
\varphi\psi
\,\textrm{d}y\,\textrm{d}\tau
\,\textrm{d}x\,\textrm{d}t
\\
+
\int_0^T\int_\Omega
\int_\mathcal{Q}
\frac{1}{a_2}
B
\Big(
     v
     \Big(\nabla b_2+\omega_{\alpha,1}
          \nabla_y\Big(\frac{b}{b_1}\Big)
     \Big)
     +
     b_2
     (
      \nabla v
      +
      \nabla_y v_1
     )
\Big)
\cdot
     \varphi
     \nabla_y\psi
\,\textrm{d}y\,\textrm{d}\tau
\,\textrm{d}x\,\textrm{d}t
=
0
\;,
\end{multline}
which is the weak formulation of \eqref{fas075}.
\qed

\begin{corollary}
\label{c:b=2a1}
Given $v\in L^2(0,T;H^1_0(\Omega))$
equation \eqref{fas075} admits a unique solution
$v_1\in L^2(\Omega_T;\Hper(\mathcal{Q}))$
with
$\int_\mathcal{Q}v_1\,\textup{d}y\textup{d}\tau=0$.
\end{corollary}

\begin{proof}
The proof of the uniqueness 
follows by standard energy estimates and Young inequality, taking into account the
linearity of the problem and the periodicity of $v_1$.
\end{proof}

\begin{theorem}
\label{t:b=2afact}
In the same hypotheses of Theorem~\ref{t:b=2aperv},
the corrector $v_1$ can be written in the factored form
\begin{equation}
\label{fas211}
v_1(x,t,y,\tau)
=
-\chi^j(x,t,y,\tau)\frac{\partial v}{\partial x_j}(x,t)
-\zeta(x,t,y,\tau)v(x,t)\,,
\end{equation}
where the cell functions $\chi^j$, $j=1,\dots,n$, and $\zeta$
are $\mathcal{Q}$--periodic, with null mean average over $\mathcal{Q}$,
and
are the unique solutions of
\begin{equation}
\label{fas213}
\frac{a_1}{b_1}
\frac{\partial \chi^j}{\partial\tau}
-
\frac{1}{a_2}
\emph{div}_y
\big(
     b_2 B \nabla_y(\chi^j-y_j)
\big)
=0
\end{equation}
and
\begin{equation}
\label{fas215}
\frac{a_1}{b_1}
\frac{\partial \zeta}{\partial\tau}
-
\frac{1}{a_2}
\emph{div}_y
\big(
     b_2 B \nabla_y\zeta
\big)
+
\frac{1}{a_2}
\emph{div}_y
\Big(
     B\Big(
           \nabla b_2 +\omega_{\alpha,1} \nabla_y\frac{b}{b_1}
      \Big)
\Big)
=0
\,.
\end{equation}
Moreover,
the system \eqref{fas070} and \eqref{fas075} can be written as the
single scale equation
\begin{equation}
\label{fas200}
q_{\emph{eff}}
\frac{\partial v}{\partial t}
-\emph{div}(B_{\emph{hom}}\nabla v)
+P_{\emph{eff}}\cdot\nabla v
+z_{\emph{eff}} v
=
f
\int_{\mathcal{S}}
\frac{\emph{d}\tau}{a_2}\,,
\end{equation}
where
\begin{align}
q_{\emph{eff}}
&
=
\int_{\mathcal{Y}}
\frac{a_1}{b_1}
\,\mathrm{d}y\,,
\label{fas205}
\\
B_{\emph{hom}}^{ij}
&
=
\int_{\mathcal{Q}}
\frac{b_2}{a_2}
B^{ik}
\partial_k(y^j-\chi^j)
\,\mathrm{d}y\,\mathrm{d}\tau
\nonumber
\\
&
\phantom{mmmmm}
=
\int_{\mathcal{Q}}
\frac{b_2}{a_2}
B^{\ell k}
\partial_k(y^j-\chi^j)
\partial_\ell(y^i-\chi^i)
\,\mathrm{d}y\,\mathrm{d}\tau
+
\int_{\mathcal{Q}}
\frac{a_1}{b_1}
\frac{\partial\chi^j}{\partial\tau}\chi^i
\,\mathrm{d}y\,\mathrm{d}\tau\,,
\label{fas210}
\\
P_{\emph{eff}}
&
=\!\!
\int_{\mathcal{Q}}
\Big[
\frac{b_2}{a_2}
B
\nabla_y\zeta
+
b_2
(
B
\nabla_y(y-\chi)
)^\dag
\nabla\frac{1}{a_2}
-
      \frac{1}{a_2}
B
      \nabla b_2
-
      \omega_{\alpha,1}
B
\nabla_y\Big(\frac{b}{a_2b_1}\Big)
\Big]
\,\mathrm{d}y\,\mathrm{d}\tau,
\label{fas220}
\\
z_{\emph{eff}}
&
=
     \int_{\mathcal{Y}}
      a_1\frac{\partial}{\partial t}\frac{1}{b_1}
      \,\mathrm{d}y
      +
     \int_{\mathcal{Q}}
      \Big[
      \emph{div}
      \Big(
      \frac{B b_2}{a_2}
      \nabla_y\zeta
      \Big)
\nonumber
\\
&
\phantom{mmmmmmmm}
      -\frac{1}{a_2}
      \emph{div}
      \Big(
      B
      \Big(
      \nabla b_2
      +\omega_{\alpha,1}
       \nabla_y\frac{b}{b_1}
      \Big)
      \Big)
      -
     b_2 B
     \nabla_y\zeta
     \cdot
     \nabla\frac{1}{a_2}
     \Big]
      \,\mathrm{d}y\, \mathrm{d}\tau
\,.
\label{fas230}
\end{align}
\end{theorem}

\par\noindent
\textit{Proof}
We first note that, by classical results
(see, i.e., \cite[Chapter~1, Section~2.2]{lions}),
equations \eqref{fas213} and \eqref{fas215}
admit a unique $\mathcal{Q}$--periodic solution with
null mean average.
Then, a standard computation shows that $v_1$ defined
in \eqref{fas211} satisfies \eqref{fas075}.
Finally, inserting \eqref{fas211} into \eqref{fas070} and performing
some algebraic computations we get equation \eqref{fas200}.

In particular, the second equality in \eqref{fas210} can be obtained
as follows. We first note that
\begin{displaymath}
B_{\emph{hom}}^{ij}
=
-\int_{\mathcal{Q}}
\frac{b_2}{a_2}
B^{\ell k}
\partial_k(\chi^j-y^j)
\partial_\ell y^i
\,\mathrm{d}y\,\mathrm{d}\tau\,
\;.
\end{displaymath}
Moreover,
from \eqref{fas213}, we have that
\begin{displaymath}
\int_{\mathcal{Q}}
\frac{b_2}{a_2}
B^{\ell k}
\partial_k(\chi^j-y^j)
\partial_\ell\chi^i
\,\mathrm{d}y\,\mathrm{d}\tau\,
+
\int_{\mathcal{Q}}
\frac{a_1}{b_1}
\frac{\partial\chi^j}{\partial\tau}\chi^i
\,\mathrm{d}y\,\mathrm{d}\tau\,
=0
\;.
\end{displaymath}
By summing the two equations above we get \eqref{fas210}.
\qed

\begin{corollary}
\label{c:b=2a2}
In the same hypotheses of Theorem~\ref{t:b=2aperv},
equation \eqref{fas200},
complemented with the boundary and initial conditions
\eqref{fas080} and \eqref{fas090},
and
the two--scale problem
\eqref{fas070}--\eqref{fas090}
admit a unique solution.
\end{corollary}

\begin{proof}
First we note that
the matrix $B_\text{hom}$ in \eqref{fas210}
is made of two parts, the first one is symmetric and by standard
calculations it is also positive definite.
On the other hand,
the second part, which is due to the presence
of the derivative with respect to the microscopic time $\tau$ in the parabolic
equation \eqref{fas213} for the cell functions $\chi^j$,
is antisymmetric.

However,
the uniqueness for equation \eqref{fas200},
complemented with \eqref{fas080} and \eqref{fas090},
still follows by standard energy estimates, Gronwall and
Young inequalities, taking into account that
the antisymmetric part of the homogenized matrix $B_\textup{hom}$
disappears in the energy estimate.
Indeed, it is multiplied by the symmetric matrix
$(v-\tilde{v})_{x_i}(v-\tilde{v})_{x_j}$, where $v$ and $\tilde{v}$
are two different solutions of \eqref{fas200}.
Thus, the estimation can be performed as usual.

To prove uniqueness for the problem
\eqref{fas070}--\eqref{fas090},
we assume that there exist two solutions
$(v,v_1)$ and $(\tilde{v},\tilde{v}_1)$.
From Corollary~\ref{c:b=2a1}
and
Theorem~\ref{t:b=2afact}
it follows that $v_1$ and $\tilde{v}_1$
are given as in \eqref{fas211}
for $v$ and $\tilde{v}$, respectively.
By substituting these two representations of $v_1$ and $\tilde{v}_1$
in \eqref{fas070}, it follows that both $v$ and $\tilde{v}$
satisfy \eqref{fas200}. Thus, by uniqueness of the solution of
\eqref{fas200}, we have that $v=\tilde{v}$ and, therefore,
we also have $v_1=\tilde{v}_1$.
\end{proof}

\begin{remark}
\label{r:simBhom}
Notice that the antisymmetric term disappears in the homogenized
matrix \eqref{fas210} under some additional assumptions.
For instance, when the matrix $B$ and the coefficients
$a_1$, $a_2$, $b_1$, and $b_2$ do not depend on the macroscopic
space variable $x$ \cite{lions}.
\end{remark}

\begin{theorem}
\label{t:b>2aperv}
Let $\beta>2\alpha$.
Then, there exist
$v\in L^2(0,T;H^1_0(\Omega))$
and
$v_1\in L^2(\Omega_T;\Hper(\mathcal{Q}))$,
with
$\int_\mathcal{Q} v_1
\,\emph{d}y\,\emph{d}\tau=0$, such that \eqref{fas030}--\eqref{fas060} hold
and
\begin{equation}
\label{fas>062}
\frac{\partial v_1}{\partial\tau}
=0.
\end{equation}
Moreover,
the pair $(v,v_1)$ is a weak solution
of the two--scale problem \eqref{fas070}, \eqref{fas080}, \eqref{fas090}, complemented with the microscale equation
\begin{multline}
  \label{fas>075}
  \emph{div}_y
  \Big[
  \mathcal{M_\mathcal{S}}\Big(\frac{B b_2}{a_2}\Big)
  (\nabla v+\nabla_y v_1)
  \\
  \phantom{mmm}
      +v
       \mathcal{M_\mathcal{S}}\Big(\frac{B \nabla b_2}{a_2}\Big)
      +\omega_{\alpha,1}v
       \mathcal{M_\mathcal{S}}\Big(\frac{B}{a_2}
                      \nabla_y\Big(\frac{b}{b_1}\Big)
                      \Big)
                      \Big]
                      =0\,,
                      \qquad
                      \emph{ in } \Omega_T\times\mathcal{Y}\,.
\end{multline}
\end{theorem}

\par\noindent
\textit{Proof}
The convergence results in \eqref{fas030} and \eqref{fas040}
still
follow from \eqref{fas100};
\eqref{fas050} follows from \eqref{fas140}
and
\cite[Proposition~2.12]{AAB2017},
with $m=1/2$
and
by
replacing
$\tau$ with $\varepsilon^{\beta}$.
Finally,
\eqref{fas060}
and
\eqref{fas>062}
follow from
\eqref{fas140}
and
\cite[Theorem~2.16]{AAB2017},
with $m=1/2$,
by
replacing
$\tau$ with $\varepsilon^{\beta}$,
$\varepsilon$ with $\varepsilon^\alpha$
and
taking into account that $\beta>2\alpha$.

The proof of \eqref{fas070} and \eqref{fas090}
is exactly the same as
in the case $\beta=2\alpha$.

In order to prove \eqref{fas>075}, we
take into account \eqref{fas>062} and choose
the test function
$\phi_\varepsilon(x,t)=
  \varepsilon^\alpha(\varphi(x,t)/a_2(x,t,t/\varepsilon^{\beta}))
                     \psi(x/\varepsilon^\alpha)$,
where $\varphi\in \mathcal{C}^\infty(\overline\Omega_T)$ with $\varphi=0$ on
$\partial\Omega\times[0,T]$, and $\psi\in \Hper(\mathcal{Y})$,
as test function in \eqref{fas023}. We unfold the resulting
equation and obtain
\begin{multline}
\label{pfas>050}
-
 \varepsilon^\alpha
 \int_0^T\int_\Omega
 \int_\mathcal{Q}
    \mathcal{T}_\varepsilon(v_\varepsilon)
    \mathcal{T}_\varepsilon\Big(\frac{1}{b^\varepsilon_1}\Big)
\Big(
     \mathcal{T}_\varepsilon\Big(\frac{\partial a^\varepsilon_1}
                                      {\partial t}\Big)
     \mathcal{T}_\varepsilon(\varphi\psi)
     +
     \mathcal{T}_\varepsilon(a^\varepsilon_1)
     \mathcal{T}_\varepsilon(\psi)
     \mathcal{T}_\varepsilon\Big(\frac{\partial \varphi}
                                      {\partial t}\Big)
\Big)
\,\textrm{d}y\,\textrm{d}\tau
\,\textrm{d}x\,\textrm{d}t
\\
+
\varepsilon^\alpha
\int_0^T\int_\Omega
\int_\mathcal{Q}
\mathcal{T}_\varepsilon(B^\varepsilon)
\Big(
     \mathcal{T}_\varepsilon(v_\varepsilon)
     \mathcal{T}_\varepsilon\Big(
           \nabla_xb^\varepsilon_2
           +\varepsilon\nabla_x\Big(\frac{b^\varepsilon}
                                         {b^\varepsilon_1}\Big)
           +\varepsilon^{1-\alpha}
                      \nabla_y\Big(\frac{b^\varepsilon}
                                        {b^\varepsilon_1}\Big)
                            \Big)
     +
     \mathcal{T}_\varepsilon\Big(b^\varepsilon_2
                                 +\varepsilon\frac{b^\varepsilon}
                                                  {b^\varepsilon_1}\Big)
     \mathcal{T}_\varepsilon(\nabla v_\varepsilon)
\Big)
\\
\cdot
\Big(
     \mathcal{T}_\varepsilon\Big(\nabla\Big(\frac{\varphi}{a_2}\Big)\Big)
     \mathcal{T}_\varepsilon(\psi)
     +
     \mathcal{T}_\varepsilon\Big(\frac{\varphi}{a_2}\Big)
     \mathcal{T}_\varepsilon(\nabla_x\psi)
     +
     \frac{1}{\varepsilon^\alpha}
     \mathcal{T}_\varepsilon\Big(\frac{\varphi}{a_2}\Big)
     \mathcal{T}_\varepsilon(\nabla_y\psi)
\Big)
\,\textrm{d}y\,\textrm{d}\tau
\,\textrm{d}x\,\textrm{d}t
\\
=
\varepsilon^\alpha
\int_0^T\int_\Omega
f \frac{\varphi}{a_2} \psi
\,\textrm{d}x\,\textrm{d}t
+
\varepsilon^\alpha
\int_\Omega
\frac{\bar{v}_\varepsilon(x)}{b^\varepsilon_1(x,0)}
a^\varepsilon_1(x,0)
\varphi(x,0)
\psi\Big(\frac{x}{\varepsilon^\alpha}\Big)
\,\textrm{d}x
+R_\varepsilon
\;,
\end{multline}
where $R_\varepsilon\to0$ for $\varepsilon\to0$.
Now, passing to the limit $\varepsilon\to0$
and taking into account \eqref{fas060},
we get
\begin{multline}
\label{pfas>060}
\int_0^T\int_\Omega
\int_\mathcal{Q}
\frac{1}{a_2}
B
\Big(
     v
     \Big(\nabla b_2+\omega_{\alpha,1}
          \nabla_y\Big(\frac{b}{b_1}\Big)
     \Big)
     +
     b_2
     (
      \nabla v
      +
      \nabla_y v_1
     )
\Big)
\cdot
     \varphi
     \nabla_y\psi
\,\textrm{d}y\,\textrm{d}\tau
\,\textrm{d}x\,\textrm{d}t
=
0
\;,
\end{multline}
which is the weak formulation of \eqref{fas>075}.
\qed

Notice that, by \eqref{fas>062}, actually $v_1\in \Hper(\mathcal{Y})$ with $\int_{\mathcal{Y}} v_1\,\textup{d}y=0$. Moreover, 
similarly to the case $\beta=2\alpha$ discussed above, we have the
following corollary.

\begin{corollary}
\label{c:b>2a1}
Given $v\in L^2(0,T;H^1_0(\Omega))$,
equation \eqref{fas>075} admits a unique solution
$v_1\in L^2(\Omega_T;\Hper(\mathcal{Y}))$
with
$\int_\mathcal{Y}v_1\,\textup{d}y=0$.
\end{corollary}

\begin{theorem}
\label{t:b>2afact}
In the same hypotheses of Theorem~\ref{t:b>2aperv},
the corrector $v_1$ can be written in the factored form
\begin{equation}
\label{fas>211}
v_1(x,t,y)
=
-\chi^j(x,t,y)\frac{\partial v}{\partial x_j}(x,t)
-\zeta(x,t,y)v(x,t)
\end{equation}
where the cell functions $\chi^j$, $j=1,\dots,n$, and $\zeta$
are $\mathcal{Y}$--periodic, with null mean average over $\mathcal{Y}$,
and
are the unique solutions of
\begin{equation}
\label{fas>213}
\emph{div}_y
\Big(
     \mathcal{M_\mathcal{S}}\Big(\frac{b_2 B}{a_2}\Big) \nabla_y(\chi^j-y_j)
\Big)
=0
\end{equation}
and
\begin{equation}
\label{fas>215}
-
\emph{div}_y
\Big(
     \mathcal{M_\mathcal{S}}\Big(\frac{b_2 B}{a_2}\Big) \nabla_y\zeta
\Big)
+
\emph{div}_y
\mathcal{M_\mathcal{S}}
\Big(
     \frac{B}{a_2}
      \Big(
           \nabla b_2 +\omega_{\alpha,1} \nabla_y\frac{b}{b_1}
      \Big)
\Big)
=0
\,.
\end{equation}
Moreover,
the system \eqref{fas070} and \eqref{fas>075} can be written as the
single scale equation
\eqref{fas200},
where
$q_{\emph{eff}}$,
$P_{\emph{eff}}$,
and
$z_{\emph{eff}}$
are formally defined as in Theorem~\ref{t:b=2afact},
and
\begin{equation}
\label{fas>210}
B_{\emph{hom}}^{ij}
=
\int_{\mathcal{Q}}
\frac{b_2}{a_2}
B^{ik}
\partial_k(y^j-\chi^j)
\,\mathrm{d}y\,\mathrm{d}\tau\,
=
\int_{\mathcal{Q}}
\frac{b_2}{a_2}
B^{\ell k}
\partial_k(y^j-\chi^j)
\partial_\ell(y^i-\chi^i)
\,\mathrm{d}y\,\mathrm{d}\tau\,,
\end{equation}
with $\chi$ and $\zeta$ being the solutions of \eqref{fas>213}
and \eqref{fas>215}.
\end{theorem}

\par\noindent
\textit{Proof}
We first note that, by classical results
(see, i.e., \cite[Chapter~1, Section~2.2]{lions}),
equations \eqref{fas>213} and \eqref{fas>215}
admit a unique $\mathcal{Y}$--periodic solution with
null mean average.
Then, a standard computation shows that $v_1$ defined
in \eqref{fas>211} satisfies \eqref{fas>075}.
Finally, inserting \eqref{fas>211} into \eqref{fas070} and performing
some algebraic computations, we get equation \eqref{fas200}.
In particular, the second equality \eqref{fas>210} is obtained
as done for \eqref{fas210} in Theorem~\ref{t:b=2afact}, by using, now, \eqref{fas>213}.
\qed

\begin{corollary}
\label{c:b>2a2}
In the same hypotheses of Theorem~\ref{t:b>2aperv},
equation \eqref{fas200},
with the homogenized matrix $B_\textup{hom}$
given in \eqref{fas>210} and
complemented with the boundary and initial conditions
\eqref{fas080} and \eqref{fas090},
admits a unique solution.
Moreover,
the two--scale problem
\eqref{fas070}, \eqref{fas>075}, \eqref{fas080}, and \eqref{fas090}
admits a unique solution.
\end{corollary}

\begin{proof}
First we note that
the matrix $B_\text{hom}$ in \eqref{fas>210}
is symmetric and, by standard
calculations, it is also positive definite.
Thus,
the uniqueness for equation \eqref{fas200},
complemented with \eqref{fas080} and \eqref{fas090},
as usual follows by standard energy estimates, Gronwall and
Young inequalities.

The second part of the corollary can be proven as we did
for Corollary~\ref{c:b=2a2}.
\end{proof}

\subsection{Slow oscillations}
\label{s:slow}
\par\noindent
In this section, we consider the remaining case $\beta<2\alpha$.

\begin{theorem}
\label{t:b<2aperv}
Let $\beta<2\alpha$.
Then, there exist
$v\in L^2(0,T;H^1_0(\Omega))$
and
$v_1\in L^2(\Omega_T\times\mathcal{S};\Hper(\mathcal{Y}))$,
with $\int_\mathcal{Y} v_1\,\emph{d}y=0$,
such that
\eqref{fas030}, \eqref{fas040} and \eqref{fas060} hold, as well as
\begin{equation}
  \label{fas<050bis}
  \mathcal{T}_\varepsilon(v_\varepsilon)
  \rightharpoonup v\,,
  \qquad\qquad
  \emph{ weakly in } L^2(\Omega_T\times \mathcal{Q}).
\end{equation}
Moreover,
the pair $(v,v_1)$ is a weak solution
of the two--scale problem \eqref{fas070}, \eqref{fas080}, \eqref{fas090}, complemented with the microscale equation
\begin{equation}
  \label{fas<075}
  \frac{1}{a_2}
  \emph{div}_y
  \Big[
  B
  \Big(
  b_2(\nabla v+\nabla_y v_1)
  +v\nabla b_2
  +\omega_{\alpha,1}v\nabla_y\Big(\frac{b}{b_1}\Big)
  \Big)
  \Big]
  =0\,,
  \qquad\qquad
  \emph{ in } \Omega_T\times\mathcal{Q}
  \,.
\end{equation}
\end{theorem}

\par\noindent
\textit{Proof}
The convergence results in \eqref{fas030} and \eqref{fas040}
follow from \eqref{fas100};
\eqref{fas060}
follows from \eqref{fas040} and \cite[Theorem~2.11]{AAB2017}.
In order to prove \eqref{fas<050bis}, we proceed as follows.
As, for instance, in \cite[Proposition~2.8]{AAB2017}, for a suitable $\widetilde v(x,t,y,\tau)$, we have
\begin{equation*}
  \mathcal{T}_{\varepsilon}(v_{\varepsilon})
  \rightharpoonup \widetilde v\,,
\qquad  \qquad
  \text{ weakly in } L^2(\Omega_T\times \mathcal{Q}),
\end{equation*}
as a consequence of \eqref{fas040}. On the other hand,
by \eqref{fas340} and \cite[Proposition~2.12]{AAB2017}, with $m=1/2$ and $\tau=\varepsilon^{\beta}$, we get that
\begin{equation*}
  \mathcal{T}_{\varepsilon}(v_{\varepsilon})
  \rightharpoonup  v\,,
  \qquad\qquad
  \text{ weakly in } L^2(\Omega\times (T_{1},T); H^{1}(\mathcal{Q})).
\end{equation*}
By testing with compactly supported functions in $\Omega_{T}\times\mathcal{Q}$ we conclude that $\widetilde v=v$.

The proof of \eqref{fas070} and \eqref{fas090}
is the same as in the case $\beta=2\alpha$.
In order to prove \eqref{fas<075}, we choose
$\phi_\varepsilon(x,t)=
  \varepsilon^\alpha(\varphi(x,t)/a_2(x,t,t/\varepsilon^{\beta}))
                     \psi(x/\varepsilon^\alpha,t/\varepsilon^{\beta})$,
where $\varphi\in \mathcal{C}^\infty(\overline\Omega_T)$ with $\varphi(x,T)=0$ in $\overline\Omega$ and $\varphi=0$ on
$\partial\Omega\times[0,T]$, and $\psi\in \Hper(\mathcal{Q})$ with $\psi(y,0)=\psi(y,1)=0$ in $\mathcal{Y}$,
as test function in \eqref{fas023}. We unfold the resulting
equation and obtain
\begin{multline}
\label{pfas<050}
-
 \varepsilon^\alpha
 \int_0^T\int_\Omega
 \int_\mathcal{Q}
    \mathcal{T}_\varepsilon(v_\varepsilon)
    \mathcal{T}_\varepsilon\Big(\frac{1}{b^\varepsilon_1}\Big)
\Big(
     \mathcal{T}_\varepsilon\Big(\frac{\partial a^\varepsilon_1}
                                      {\partial t}\Big)
     \mathcal{T}_\varepsilon(\varphi\psi)
     +
     \mathcal{T}_\varepsilon(a^\varepsilon_1)
     \mathcal{T}_\varepsilon(\psi)
     \mathcal{T}_\varepsilon\Big(\frac{\partial \varphi}
                                      {\partial t}\Big)
\Big)
\,\textrm{d}y\,\textrm{d}\tau
\,\textrm{d}x\,\textrm{d}t
\\
-
 \varepsilon^{\alpha-\beta}
 \int_0^T\int_\Omega
 \int_\mathcal{Q}
    \mathcal{T}_\varepsilon(v_\varepsilon)
    \mathcal{T}_\varepsilon\Big(\frac{1}{b^\varepsilon_1}\Big)
     \mathcal{T}_\varepsilon(a^\varepsilon_1)
     \mathcal{T}_\varepsilon(\varphi)
     \mathcal{T}_\varepsilon\Big(\frac{\partial\psi}{\partial\tau}\Big)
\,\textrm{d}y\,\textrm{d}\tau
\,\textrm{d}x\,\textrm{d}t
\\
+
\varepsilon^\alpha
\int_0^T\int_\Omega
\int_\mathcal{Q}
\mathcal{T}_\varepsilon(B^\varepsilon)
\Big(
     \mathcal{T}_\varepsilon(v_\varepsilon)
     \mathcal{T}_\varepsilon\Big(
           \nabla_xb^\varepsilon_2
           +\varepsilon\nabla_x\Big(\frac{b^\varepsilon}
                                         {b^\varepsilon_1}\Big)
           +\varepsilon^{1-\alpha}
                      \nabla_y\Big(\frac{b^\varepsilon}
                                        {b^\varepsilon_1}\Big)
                            \Big)
     +
     \mathcal{T}_\varepsilon\Big(b^\varepsilon_2
                                 +\varepsilon\frac{b^\varepsilon}
                                                  {b^\varepsilon_1}\Big)
     \mathcal{T}_\varepsilon(\nabla v_\varepsilon)
\Big)
\\
\cdot
\Big(
     \mathcal{T}_\varepsilon\Big(\nabla\Big(\frac{\varphi}{a_2}\Big)\Big)
     \mathcal{T}_\varepsilon(\psi)
     +
     \mathcal{T}_\varepsilon\Big(\frac{\varphi}{a_2}\Big)
     \mathcal{T}_\varepsilon(\nabla_x\psi)
     +
     \frac{1}{\varepsilon^\alpha}
     \mathcal{T}_\varepsilon\Big(\frac{\varphi}{a_2}\Big)
     \mathcal{T}_\varepsilon(\nabla_y\psi)
\Big)
\,\textrm{d}y\,\textrm{d}\tau
\,\textrm{d}x\,\textrm{d}t
\\
=
\varepsilon^\alpha
\int_0^T\int_\Omega
f \frac{\varphi}{a_2} \psi
\,\textrm{d}x\,\textrm{d}t
+R_\varepsilon
\;,
\end{multline}
where $R_\varepsilon\to0$ for $\varepsilon\to0$.

Using that, as we show below,
the second term in \eqref{pfas<050}
tends to zero in the limit $\varepsilon\to0$,
we get the weak formulation of \eqref{fas<075},
similarly as we did for \eqref{pfas>060} in the proof of Theorem~\ref{t:b>2aperv}.

Indeed,
the second term in \eqref{pfas<050}
can be written as
\begin{multline}
\label{pfas<120}
 \varepsilon^{2\alpha-\beta}
 \int_0^T\int_\Omega
 \int_\mathcal{Q}
\frac{1}{\varepsilon^\alpha}
    \mathcal{Z}_\varepsilon(v_\varepsilon)
    \mathcal{T}_\varepsilon\Big(\frac{1}{b^\varepsilon_1}\Big)
     \mathcal{T}_\varepsilon(a^\varepsilon_1)
     \mathcal{T}_\varepsilon(\varphi)
     \mathcal{T}_\varepsilon\Big(\frac{\partial\psi}{\partial\tau}\Big)
\,\textrm{d}y\,\textrm{d}\tau
\,\textrm{d}x\,\textrm{d}t
\\
+
 \varepsilon^{\alpha-\beta}
 \int_0^T\int_\Omega
 \int_\mathcal{Q}
    \mathcal{M}_\varepsilon(v_\varepsilon)
    \mathcal{T}_\varepsilon\Big(\frac{1}{b^\varepsilon_1}\Big)
     \mathcal{T}_\varepsilon(a^\varepsilon_1)
     \mathcal{T}_\varepsilon(\varphi)
     \mathcal{T}_\varepsilon\Big(\frac{\partial\psi}{\partial\tau}\Big)
\,\textrm{d}y\,\textrm{d}\tau
\,\textrm{d}x\,\textrm{d}t
=
J_1^\varepsilon
+
J_2^\varepsilon
\;.
\end{multline}

Recalling
\cite[Proposition~2.22]{AAB2017}
(with, in the notation there,
$m=r=1/2$,
$\alpha=1$,
$\varepsilon$ replaced by $\varepsilon^\alpha$,
and
$\tau$ replaced by $\varepsilon^\beta$),
it follows that $J_1^\varepsilon\to0$, for $\varepsilon\to0$.
Moreover,
we write
\begin{multline}
\label{pfas<110}
J_2^\varepsilon
=
 \varepsilon^{2\alpha-\beta}
 \int_0^T\int_\Omega
 \int_\mathcal{Q}
    \mathcal{M}_\varepsilon(v_\varepsilon)
    \mathcal{T}_\varepsilon\Big(\frac{1}{b^\varepsilon_1}\Big)
     \mathcal{T}_\varepsilon(a^\varepsilon_1)
     \frac{1}{\varepsilon^\alpha}
     \mathcal{Z}_\varepsilon(\varphi)
     \mathcal{T}_\varepsilon\Big(\frac{\partial\psi}{\partial\tau}\Big)
\,\textrm{d}y\,\textrm{d}\tau
\,\textrm{d}x\,\textrm{d}t
\\
+
 \varepsilon^{\alpha-\beta}
 \int_0^T\int_\Omega
 \int_\mathcal{Q}
    \mathcal{M}_\varepsilon(v_\varepsilon)
    \mathcal{T}_\varepsilon\Big(\frac{1}{b^\varepsilon_1}\Big)
     \mathcal{T}_\varepsilon(a^\varepsilon_1)
     \mathcal{M}_\varepsilon(\varphi)
     \mathcal{T}_\varepsilon\Big(\frac{\partial\psi}{\partial\tau}\Big)
\,\textrm{d}y\,\textrm{d}\tau
\,\textrm{d}x\,\textrm{d}t
\end{multline}
and note that,
using
\cite[Remark~2.23]{AAB2017}
(with, in the notation there,
$m=r=1/2$,
$\alpha=1$,
$\varepsilon$ replaced by $\varepsilon^\alpha$,
and
$\tau$ replaced by $\varepsilon^\beta$),
the first term tends to zero, in the limit
$\varepsilon\to0$.

Now, we integrate the second term in \eqref{pfas<110} with respect
to $\tau$ taking into account that
\begin{displaymath}
     \mathcal{T}_\varepsilon\Big(\frac{\partial\psi}{\partial\tau}\Big)
=
     \frac{\partial}{\partial\tau} \mathcal{T}_\varepsilon(\psi)
\end{displaymath}
and we get that it is equal to
\begin{multline}
\label{pfas<130}
-
 \varepsilon^{\alpha}
 \int_0^T\int_\Omega
 \int_\mathcal{Q}
    \mathcal{M}_\varepsilon(v_\varepsilon)
    \frac{1}{\varepsilon^\beta}
    \frac{\partial}{\partial\tau}\Big(
    \mathcal{T}_\varepsilon\Big(\frac{a^\varepsilon_1}{b^\varepsilon_1}\Big)
                                 \Big)
     \mathcal{M}_\varepsilon(\varphi)
     \mathcal{T}_\varepsilon(\psi)
\,\textrm{d}y\,\textrm{d}\tau
\,\textrm{d}x\,\textrm{d}t
\\
=
-
 \varepsilon^{\alpha}
 \int_0^T\int_\Omega
 \int_\mathcal{Q}
    \mathcal{M}_\varepsilon(v_\varepsilon)
    \mathcal{T}_\varepsilon\Big(
    \frac{\partial}{\partial t}
    \Big(\frac{a^\varepsilon_1}{b^\varepsilon_1}\Big)\Big)
     \mathcal{M}_\varepsilon(\varphi)
     \mathcal{T}_\varepsilon(\psi)
\,\textrm{d}y\,\textrm{d}\tau
\,\textrm{d}x\,\textrm{d}t
\to0
\;\;,
\end{multline}
where we
used that
\begin{displaymath}
    \frac{1}{\varepsilon^\beta}
    \frac{\partial}{\partial\tau}\Big(
    \mathcal{T}_\varepsilon\Big(\frac{a^\varepsilon_1}{b^\varepsilon_1}\Big)
                                 \Big)
=
    \mathcal{T}_\varepsilon\Big(
    \frac{\partial}{\partial t}
    \Big(\frac{a^\varepsilon_1}{b^\varepsilon_1}\Big)\Big)
\;.
\end{displaymath}
\qed

Similarly to the cases $\beta\ge2\alpha$ discussed above, we have the
following corollary.

\begin{corollary}
\label{c:b<2a1}
Given $v\in L^2(0,T;H^1_0(\Omega))$,
equation \eqref{fas<075} admits a unique solution
$v_1\in L^2(\Omega_T\times\mathcal{S};\Hper(\mathcal{Y}))$
with
$\int_\mathcal{Y}v_1\,\textup{d}y=0$.
\end{corollary}

\begin{theorem}
\label{t:b<2afact}
In the same hypotheses of Theorem~\ref{t:b<2aperv},
the corrector $v_1$ can be written in the factored form
\eqref{fas211},
where the cell functions $\chi^j$, $j=1,\dots,n$ and $\zeta$
are $\mathcal{Y}$--periodic, with null mean average over $\mathcal{Y}$,
and
are the unique solutions of
\begin{equation}
\label{fas<213}
\frac{1}{a_2}
\emph{div}_y
\big(
     b_2 B \nabla_y(\chi^j-y_j)
\big)
=0
\end{equation}
and
\begin{equation}
\label{fas<215}
\frac{1}{a_2}
\emph{div}_y
\big(
     b_2 B \nabla_y\zeta
\big)
+
\frac{1}{a_2}
\emph{div}_y
\Big(
     B\Big(
           \nabla b_2 +\omega_{\alpha,1} \nabla_y\frac{b}{b_1}
      \Big)
\Big)
=0
\,.
\end{equation}
Moreover,
the system \eqref{fas070} and \eqref{fas<075} can be written as the
single scale equation
\eqref{fas200},
where
$q_{\emph{eff}}$,
$P_{\emph{eff}}$,
and
$z_{\emph{eff}}$
are formally defined as in Theorem~\ref{t:b=2afact},
and $B_{\emph{hom}}$
is defined as in \eqref{fas>210},
with $\chi$ and $\zeta$ being the solutions of \eqref{fas<213}
and \eqref{fas<215}.
\end{theorem}

\par\noindent
\textit{Proof}
As above, by classical results
\cite[Chapter~1, Section~2.2]{lions},
equations \eqref{fas<213} and \eqref{fas<215}
admit a unique $\mathcal{Y}$--periodic solution, with
null mean average.
Then, a standard computation shows that $v_1$ defined
in \eqref{fas211} satisfies \eqref{fas>075}.
Finally, inserting \eqref{fas211} into \eqref{fas070} and performing
some algebraic computations, we get equation \eqref{fas200}.
\qed
\medskip

Similarly to the case $\beta>2\alpha$ discussed above, we have the
following corollary.

\begin{corollary}
\label{c:b<2a2}
In the same hypotheses of Theorem~\ref{t:b<2aperv},
equation \eqref{fas200},
with the homogenized matrix $B_\textup{hom}$
given in \eqref{fas>210}, where
$\chi$ and $\zeta$ are the solutions of \eqref{fas<213}
and \eqref{fas<215},
and
complemented with the boundary and initial conditions
\eqref{fas080} and \eqref{fas090},
admits a unique solution.
Moreover,
the two--scale problem
\eqref{fas070}, \eqref{fas<075}, \eqref{fas080}, and \eqref{fas090}
admits a unique solution.
\end{corollary}

Notice that, in the present case, i.e. $\beta<2\alpha$,
the dependence of the cell functions $\chi$ and $\zeta$
on the microtime $\tau$
is only parametric (as well as on $(x,t)$), via the coefficients
of the corresponding equations.

\subsection{Proof of Theorem~\ref{t:b=2a}.}
\label{s:main_dim}
\par\noindent

In the case $\beta=2\alpha$, by Theorems \ref{t:b=2aperv} and \ref{t:b=2afact},
we obtain that $v_\varepsilon\rightharpoonup v$ weakly in $L^2(\Omega_T)$,
where $v$ is the solution of \eqref{fas200}.
By using \eqref{fas100} and \eqref{eq:time_comp}, it follows that the convergence is, indeed,
strong in $L^2(\Omega_T)$.
Moreover, by the assumptions on $b_1$, it follows that
$1/b_1^\varepsilon\rightharpoonup \int_\mathcal{Y}\textrm{d}y/b_1$
weakly$^*$ in $L^\infty(\Om_T)$.
Therefore,
\begin{equation}
\label{fas400}
u_\varepsilon
=\frac{v_\varepsilon}{b_1^\varepsilon}
\rightharpoonup
v \int_\mathcal{Y}\frac{\textrm{d}y}{b_1}
=:u
\;.
\end{equation}
By replacing $v=u\big(\int_\mathcal{Y}\textrm{d}y/b_1\big)^{-1}$ in
\eqref{fas200}, we eventually get \eqref{b=2a} and \eqref{b=2aini}.

The uniqueness of the solution $u$ of \eqref{b=2a}--\eqref{b=2aini}
follows by
the uniqueness for equation \eqref{fas200}, complemented with the boundary and the initial
conditions \eqref{fas080} and \eqref{fas090}.

The cases $\beta\not=2\alpha$ are treated in the same way, of course by appealing to
Theorems \ref{t:b>2aperv} and \ref{t:b=2afact}, for $\beta>2\alpha$ and Theorems \ref{t:b<2aperv} and \ref{t:b<2afact},
for $\beta<2\alpha$, respectively.

\qed

\subsection{Some particular cases}
\label{s:spec}
\par\noindent
Here we discuss some very special cases in which the
upscaled equations take specific forms.

\subsubsection{Pure product case}
\label{s:ppc}
\par\noindent
In the case $b=0$, we can fix $\alpha=1$ without loss of generality,
since no other scaling, excepted $x/\varepsilon^\alpha$ and
$t/\varepsilon^\beta$, is present in the equation.

The homogenized equations for the limit function $u$, appearing
in Theorem~\ref{t:b=2a},
take the form
\begin{multline}
\int_{\mathcal{Q}}
\Big[
\frac{a_1}{\int_{\mathcal{S}} a_2^{-1}\,\textup{d}\tau}
\frac{\partial}{\partial t}
\Big(
     \frac{u}{b_1\int_\mathcal{Y}b_1^{-1}\,\textup{d}y}
\Big)
-
\frac{1}{a_2\int_{\mathcal{S}}a_2^{-1}\,\textup{d}\tau}
\textup{div}
\Big(
B_{\textup{eff}}
\nabla\Big(b_2
\frac{u}{\int_\mathcal{Y}b_1^{-1}\,\textup{d}y}\Big)
\Big)
\\
-
\frac{1}{a_2\int_{\mathcal{S}} a_2^{-1}\,\textup{d}\tau}
\textup{div}
\Big(
B
\nabla_y\Big(
             -b_2\zeta
             +\chi\cdot\nabla b_2
        \Big)
    \frac{u}{\int_\mathcal{Y}b_1^{-1}\,\textup{d}y}
\Big)
\Big]
\,\textup{d}y\,\textup{d}\tau
=
f\,,
\qquad
\textup{ in } \Omega_T\,,
\label{b=2a-sp}
\end{multline}
where
$B_{\emph{eff}} = B \nabla_y(y-\chi)$
and the cell functions $\chi$ and $\zeta$ satisfy
\eqref{fas213} and \eqref{fas215} with $b=0$ for $\beta=2$,
\eqref{fas>213} and \eqref{fas>215} with $b=0$ for $\beta>2$,
\eqref{fas<213} and \eqref{fas<215} with $b=0$ for $\beta<2$.

Incidentally, this is also the case when $\alpha<1$ even if $b\neq0$,
which means that the space oscillation is greater than the
non--product perturbation.

\subsubsection{Pure Fick case}
\label{s:pfc}
\par\noindent
In the case $b_1=b_2=1$ and $b=0$, as above
we can fix $\alpha=1$ without loss of generality,
In such a case, our problem is a particular case of the
one studied in \cite{AAB2017}, with the time oscillation
being a power of the space oscillation.

Indeed,
it is easy to prove that the cell function $\zeta$ is equal to zero
for any $\beta$
and
the cell function $\chi$ satisfies
\cite[equation~(7.1)]{AAB2017} for $\beta=2$,
\cite[equation~(7.2)]{AAB2017} for $\beta>2$,
and
\cite[equation~(7.3)]{AAB2017} for $\beta<2$,
respectively.

Moreover, the homogenized equation \eqref{b=2a} reduces to
\cite[equation~(7.4)]{AAB2017}, for any choice of $\beta$.
In particular, when the capacity is independent of the macrovariables,
the resulting equation turns to be the pure Fick equation
\begin{equation}
\label{fick:fin}
\frac{\int_{\mathcal{Y}} a_1\,\textup{d}y}
     {\int_{\mathcal{S}} a_2^{-1}\,\textup{d}\tau}
u_t
-
\Div
\Big(
\frac{\int_{\mathcal{Q}} (B_\text{eff}/a_2)\,\textup{d}y\,\textup{d}\tau}
{\int_{\mathcal{S}}a_2^{-1}\,\textup{d}\tau
}
\nabla u
\Big)
=
f
,
\end{equation}
where the capacity and the diffusion matrix appear mixed in the upscaled
diffusion coefficient.

\subsubsection{Pure Fokker--Planck case}
\label{s:pfpc}
\par\noindent
If $B$ is the identity matrix, it follows that $\chi$ is always
identically zero, so that $B_\textup{eff}=B$, and by periodicity
\begin{equation}
\label{pfpc000}
\int_{\mathcal{Y}}
\nabla_y\Big(
             \omega_{\alpha,1}\frac{b}{b_1}-b_2\zeta
        \Big)
\,\textup{d}y
=0
\,.
\end{equation}
Thus,
the limit equation reduces to
\begin{equation}
\int_{\mathcal{Q}}
\Big[
\frac{a_1}{\int_{\mathcal{S}} a_2^{-1}\,\textup{d}\tau}
\frac{\partial}{\partial t}
\Big(
     \frac{u}{b_1\int_\mathcal{Y}b_1^{-1}\,\textup{d}y}
\Big)
-
\frac{1}{a_2\int_{\mathcal{S}} a_2^{-1}\,\textup{d}\tau}
\Delta
\Big(b_2
\frac{u}{\int_\mathcal{Y}b_1^{-1}\,\textup{d}y}\Big)
\Big]
\,\textup{d}y\,\textup{d}\tau
=
f\,,\qquad
\textup{ in } \Omega_T
,
\label{b=2a-pfpc}
\end{equation}
which does not depend on the non--product perturbation $b$.
We remark that this is also valid under the milder hypothesis
that $B$ does not depend on $y$.

If in addition $\alpha<1$ or $b=0$, then also the cell function
$\zeta=0$ and therefore \eqref{pfpc000} is trivially satisfied.

The limit equation in the pure Fokker--Planck case has been
written in the from \eqref{b=2a-pfpc} to make it
as close as possible to the starting Fokker--Planck problem.
However, it is possible to formally reduce it to a standard parabolic
equation with lower order terms, in which the coefficients are
expressed in terms of the mean value on $\mathcal{Q}$ of the
coefficients of the original equation, i.e., $a_1$, $a_2$, $b_1$, and
$b_2$.

Finally, we remark that in the very particular case in which
the coefficients $a_1$, $a_2$, $b_1$, and $b_2$,
do not depend on the macroscopic variables, the equation
\eqref{b=2a-pfpc}
becomes
\begin{equation}
\label{eq:fin}
\frac{\int_{\mathcal{Y}} (a_1/b_1)\,\textup{d}y}
     {\int_{\mathcal{S}} a_2^{-1}\,\textup{d}\tau
     \int_\mathcal{Y}b_1^{-1}\,\textup{d}y}
u_t
-
\Delta
\Big(
\frac{\int_{\mathcal{S}} (b_2/a_2)\,\textup{d}\tau}
{\int_\mathcal{Y}b_1^{-1}\,\textup{d}y
 \int_{\mathcal{S}} a_2^{-1}\,\textup{d}\tau}
u
\Big)
=
f\,,\qquad\qquad
\textup{ in } \Omega_T
,
\end{equation}
which shows that, even in such a special case, the capacity and the
Fokker coefficient are mixed in the upscaled equation.

\end{document}